\documentclass[letterpaper,fleqn,11pt]{article}

\usepackage[margin=1.5in]{geometry}
\usepackage{marginnote}

\usepackage{setspace}

\usepackage{mathptmx}
\usepackage[T1]{fontenc}
\usepackage{euscript}

\usepackage{latexsym,epsfig,verbatim}
\usepackage{amsmath,amsthm,amssymb}
\usepackage{colonequals}
\usepackage{units}

\usepackage{url}
\usepackage{color}
\usepackage[colorlinks,citecolor=blue,linkcolor=red]{hyperref}
\usepackage[nameinlink]{cleveref}
\usepackage{comment}

\usepackage{accents}

\usepackage{tikz}
\usetikzlibrary{matrix,arrows,decorations.pathmorphing}

\theoremstyle{plain}
\newtheorem{theorem}{Theorem}[section]
\newtheorem{lemma}[theorem]{Lemma}
\newtheorem{corollary}[theorem]{Corollary}
\newtheorem*{thm:quasiconvex}{\Cref{prop: geos_fellow_travel}}
\newtheorem*{thm:application}{\Cref{th: v_free_examples}}

\newtheorem{claim}[theorem]{Claim}

\newtheorem{proposition}[theorem]{Proposition}

\newtheorem{definition}[theorem]{Definition}

\theoremstyle{definition}
\newtheorem*{remark}{Remark}

\newtheorem{example}[theorem]{Example}
\newtheorem{convention}[theorem]{Convention}

\newtheorem*{case}{Case}

\newcommand{\define}[1]{\textbf{#1}}

\newcommand{\R}{\mathbb{R}}
\newcommand{\Z}{\mathbb{Z}}

\newcommand{\N}{\mathbb{N}}
\DeclareMathOperator{\diam}{diam}
\newcommand{\floor}[1]{\left\lfloor#1\right\rfloor}
\newcommand{\ceil}[1]{\left\lceil#1\right\rceil}
\newcommand{\norm}[1]{\left\Vert#1\right\Vert}
\newcommand{\abs}[1]{\left\vert#1\right\vert}
\newcommand{\inv}{^{-1}}

\DeclareMathOperator{\Teich}{Teich} 
\newcommand{\T}{\Teich}
\DeclareMathOperator{\Mod}{Mod}

\newcommand{\C}{{\EuScript C}} 

\DeclareMathOperator{\Out}{Out}
\DeclareMathOperator{\Aut}{Aut}

\DeclareMathOperator{\rank}{rk}
\newcommand{\free}{\mathbb{F}} 
\newcommand{\factor}{{\EuScript F}} 
\newcommand{\F}{\factor} 
\renewcommand{\int}{\mathcal{I}}
\newcommand{\fc}{\factor} 
\newcommand{\cv}{\hat{\os}} 
\newcommand{\os}{{\EuScript X}} 
\newcommand{\X}{\os} 
\newcommand{\dsym}{d^{\mathrm{sym}}_\os} 
\newcommand{\fproj}{\pi} 
\newcommand{\proj}{\mathrm{Pr}} 
\newcommand{\p}{\proj}

\newcommand{\rose}{{\EuScript R}} 
\newcommand{\bund}{\mathcal{E}} 
\newcommand{\base}{\mathcal{B}} 
\DeclareMathOperator{\Lip}{Lip} 
\DeclareMathOperator{\vol}{vol} 
\newcommand{\sym}{{\sf M}} 

\newcommand{\foldim}{\textrm{Im}^\text{f}}
\newcommand{\scaleim}{\textrm{Im}^\text{sc}}

\newcommand{\nbhd}[2]{{\EuScript N}_{#1}(#2)} 
\newcommand{\dhaus}{d_{\mathrm{Haus}}} 

\newcommand{\I}{\mathbf{I}}
\newcommand{\Ipl}{\I_+}
\newcommand{\Imin}{\I_-}
\newcommand{\J}{\mathbf{J}}
\newcommand{\Jpl}{\J_+}
\newcommand{\Jmin}{\J_-}
\newcommand{\cay}[2]{\mathrm{Cay}({#2}, {#1})} 

\begin{document}

\title{\textbf{\Large Hyperbolic extensions of free groups}}
\author{Spencer Dowdall and Samuel J. Taylor \thanks{
The first author was partially supported by the NSF postdoctoral fellowship, NSF MSPRF no. 1204814.
The second author was partially supported by a department fellowship from the University of Texas at Austin and by the NSF postdoctoral fellowship, NSF MSPRF no. 1400498. 
Both authors acknowledge support from U.S. National Science Foundation grants DMS 1107452, 1107263, 1107367 "RNMS: GEometric structures And Representation varieties" (the GEAR Network).
}}
\date{\today}


\maketitle

\begin{abstract}
Given a finitely generated subgroup $\Gamma \le \Out(\free)$ of the outer automorphism group of the rank $r$ free group $\free = F_r$, there is a corresponding free group extension \linebreak $1 \to \free \to E_{\Gamma} \to \Gamma \to 1$. We give sufficient conditions for when the extension $E_{\Gamma}$ is hyperbolic. In particular, we show that if all infinite order elements of $\Gamma$ are atoroidal and the action of $\Gamma$ on the free factor complex of $\free$ has a quasi-isometric orbit map, then $E_{\Gamma}$ is hyperbolic. As an application, we produce examples of hyperbolic $\free$--extensions $E_{\Gamma}$ for which $\Gamma$ has torsion and is not virtually cyclic. The proof of our main theorem involves a detailed study of  quasigeodesics in Outer space that make progress in the free factor complex. This may be of independent interest.
\end{abstract}

\setcounter{tocdepth}{1}
\begin{spacing}{0.9}\tableofcontents\end{spacing}

\section{Introduction}
Let $\free = F_r$ denote the free group of rank $r \ge 3$ and consider its group $\Out(\free)$ of outer automorphisms. These groups fit into the short exact sequence
\[1 \longrightarrow \free \overset{i}{\longrightarrow} \Aut(\free) \overset{p}{\longrightarrow} \Out(\free) \longrightarrow 1,  \]
where $a \in \free$ is mapped to its corresponding inner automorphism $i_a$ defined by $x\mapsto axa^{-1}$ for $x\in \free$.
Hence, for any $\Gamma \le \Out(\free)$ we obtain the following extension of $\free$:
\[1 \longrightarrow \free \overset{i}{\longrightarrow} E_{\Gamma} \overset{p}{\longrightarrow} \Gamma \longrightarrow 1, \]
where $E_{\Gamma}$ is equal to the preimage $p^{-1}(\Gamma) \le \Aut(\free)$.
In fact, any extension of $\free$ induces a homomorphism to $\Out(\free)$ and thereby produces an extension of the above form (see \Cref{sec: extensions} for details). 
This paper will address the following question:\\ 

\emph{What conditions on $\Gamma \le \Out(\free)$ imply that the extension $E_{\Gamma}$ is a hyperbolic group? }\\

This question fits in to a long history of understanding hyperbolic group extensions that goes back to Thurston's work on the hyperbolization of fibered $3$--manifolds. From a group-theoretic perspective, the Bestvina--Feighn combination theorem \cite{BF92} provides a combinatorial framework to understand the structure of more general hyperbolic group extensions. Using this, Farb and Mosher's influential work \cite{FarbMosher} initiated the systematic study of hyperbolic extensions of surface groups (see  \Cref{sub: motivation}). Our answer to the question above continues this investigation in the setting of free group extensions. 

\subsection{Statements of results}
To state our main theorem, we briefly recall the relevant definitions and refer the reader to \Cref{sec: prelims} for additional details. First, an outer automorphism $\phi\in \Out(\free)$ is \define{atoroidal}, or \define{hyperbolic}, if no power of $\phi$ fixes any nontrivial conjugacy class in $\free$. 
Similarly, $\phi\in \Out(\free)$ is \define{fully irreducible} if no power of $\phi$ preserves the conjugacy class of any proper free factor of $\free$.
The \define{(free) factor complex} $\F$ for the free group $\free$ is the simplicial complex in which each $k$ simplex corresponds to a set $[A_0], \ldots, [A_k]$ of $k+1$ conjugacy classes of proper free factors of $\free$ with properly nested representatives: $A_0 < \dotsb < A_k$.
Note that there is an obvious simplicial action $\Out(\free) \curvearrowright \F$.
 We prove the following:

\begin{theorem}\label{th: intro_main_1}
Suppose that each infinite-order element of a finitely generated subgroup $\Gamma \le \Out(\free)$ is atoroidal and that some  orbit map $\Gamma \to \F$ is a quasi-isometric embedding. Then the free group extension $E_{\Gamma}$ is hyperbolic.
\end{theorem}

\begin{remark}
Bestvina and Feighn have proven that the factor complex $\F$ is hyperbolic \cite{BFhyp}. Hence, the hypotheses of \Cref{th: intro_main_1} additionally imply that the subgroup $\Gamma$ is itself hyperbolic and that all infinite-order elements of $\Gamma$ are fully irreducible. See \Cref{sec: factor_complex} for details.
\end{remark}

\Cref{th: intro_main_1} provides combinatorial conditions on a subgroup $\Gamma \le \Out(\free)$ which guarantee that the corresponding extension $E_{\Gamma}$ is hyperbolic. This is similar to the better understood situation of hyperbolic extensions of surface groups. For surface group extensions, it follows from work of Farb--Mosher \cite{FarbMosher}, Kent--Leininger \cite{KentLein}, and Hamenst\"adt \cite{H}, that a subgroup $H$ of the mapping class group induces a hyperbolic extension of the surface group if and only if $H$ admits a quasi-isometric embedding into the curve complex of the surface.  See \Cref{sub: motivation} for details. 

\begin{remark}
Unlike the surface group case (c.f., \Cref{thm:mod_CC} below), the converse to \Cref{th: intro_main_1} does not hold: there exits subgroups $\Gamma\le\Out(\free)$ for which $E_\Gamma$ is hyperbolic but $\Gamma$ does not quasi-isometrically embed into $\fc$. For example, Brinkmann's \Cref{brink} below \cite{Brink} shows that any $\phi\in \Out(\free)$ that is atoroidal but not fully irreducible generates a cyclic subgroup of this form. 
\end{remark}

The proof of \Cref{th: intro_main_1} requires several steps and is completed in \Cref{sec: hyp_bundle} (see \Cref{for: main_result}). The first of these steps is to show that the assumption that the orbit map $\Gamma \to \F$ is a quasi-isometric embedding implies a strong quasiconvexity property for the orbit of $\Gamma$ in Outer space $\os$, the space of $\free$--marked metric graphs. This follows from our next main result, \Cref{prop: stab_intro} below, which says that quasigeodesics in Outer space that make definite progress in the factor complex are stable. For the statement, the injectivity radius of $G\in \os$ is the length of the shortest loop in the marked metric graph $G$, and the $\epsilon$--thick part $\os_\epsilon$ is the set of points with injectivity radius at least $\epsilon$. Additionally,  $\pi \colon \X \to \F$ denotes the (coarse) map that associates to each marked graph $G \in \os$ the collection $\pi(G)$ of nontrivial free factors that arise as the fundamental group of a proper subgraph of $G$.

\begin{thm:quasiconvex}

Let $\gamma\colon \I\to \os$ be a $K$--quasigeodesic  whose projection $\pi \circ \gamma\colon \I\to \fc$ is also a $K$--quasigeodesic. Then there exist constants $A,\epsilon > 0$ and $K'\ge 1$ depending only on $K$ (and the injectivity radius of the terminal endpoint $\gamma(\Ipl)$ when $\Ipl<\infty$) with the following property: If $\rho\colon \J\to \os$ is any geodesic with the same endpoints as $\gamma$, then
\begin{itemize}
\item[(i)] $\gamma(\I), \rho(\J)\subset \os_{\epsilon}$,
\item[(ii)] $d_{\mathrm{Haus}}(\gamma(\I),\rho(\J)) < A$, and
\item[(iii)] $\pi\circ \rho\colon \J\to \fc$ is a (parameterized) $K'$--quasigeodesic.
\end{itemize} 

\end{thm:quasiconvex}
In the statement of \Cref{prop: stab_intro}, $\gamma$ and $\rho$ are directed (quasi)geodesics with respect to the asymmetric Lipschitz metric $d_\os$ on Outer space, and $d_{\mathrm{Haus}}$ denotes the Hausdorff distance with respect to the symmetrized Lipschitz distance; see \Cref{sec: prelims} for a more detailed discussion of this terminology. \Cref{prop: stab_intro} is analogous to Hamenst\"{a}dt's stability theorem for quasigeodesics in Teichm\"{u}ller space that make definite progress in the curve complex \cite{Hamstability}.  \\

\Cref{th: intro_main_1} allows one to easily construct hyperbolic extensions of free groups using ping-pong arguments on hyperbolic $\Out(\free)$--graphs. For example, we can recover (\Cref{thm:free_construction}) the theorem of Bestvina--Feighn--Handel \cite{BFHlam} which states that if $f_1, \ldots,f_k$ are atoroidal, fully irreducible elements of $\Out(\free)$, then for all sufficiently large $N\ge 1$ the extension $E_{\Gamma}$ is hyperbolic for $\Gamma = \langle f_1^N, \ldots, f_k^N \rangle \le \Out(\free)$. (In \cite{BFHlam}, this is proven for $k=2$.) Further, we use \Cref{th: intro_main_1} to construct the first examples of hyperbolic free group extensions $E_{\Gamma}$ for which $\Gamma \le \Out(\free)$ has torsion and is not virtually cyclic. First, say that $f \in \Out(\free)$ is independent for a finite subgroup $H \le \Out(\free)$ if $f$ and $hfh^{-1}$ have no common powers for each $h \in H \setminus 1$.  We prove the following:

\begin{thm:application}
Let $H$ be a finite subgroup of $\Out(\free)$ and let $f \in \Out(\free)$ be a hyperbolic, fully irreducible outer automorphisms that is independent for $H$. Then for all sufficiently large $N \ge 1$, the subgroup
 \[ \Gamma = \langle H, f^N \rangle \]
 is isomorphic to $H* \Z$ and the $\free$-by-$(H*\Z)$ extension $E_{\Gamma}$ is hyperbolic.

\end{thm:application}

\subsection{Motivation from surface group extensions and some previous results} \label{sub: motivation}
In \cite{FarbMosher}, Farb and Mosher introduced  convex cocompact subgroups of $\Mod(S)$, the mapping class group of an orientable surface $S$. We will discus the case where $S$ is further assumed to be closed. A finitely generated subgroup $\Gamma \le \Mod(S)$ is \define{convex cocompact} if for some (any) $x \in \T(S)$, the Teichm\"{u}ller space of the surface $S$, the orbit $\Gamma \cdot x \subset \T(S)$ is quasiconvex with respect to the Teichm\"{u}ller metric. (See the papers of Farb--Mosher \cite{FarbMosher} and Kent--Leininger \cite{KentLein, KLsurvey} for definitions and details). Similar to the situation described above, a subgroup $\Gamma \le \Mod(S)$ gives rise to a surface group extension 
\[1 \longrightarrow \pi_1(S) \longrightarrow E_{\Gamma} \longrightarrow \Gamma \longrightarrow 1. \]
Farb and Mosher show that if $E_{\Gamma}$ is hyperbolic then $\Gamma$ is convex cocompact. Moreover, they prove that if $\Gamma$ is assumed to be free, then convex cocompactness of $\Gamma$ implies that the extension $E_{\Gamma}$ is hyperbolic \cite{FarbMosher}. The assumption that $\Gamma$ is free was later removed by Hamenst\"{a}dt in \cite{H}. Hence, the surface group extension $E_\Gamma$ is hyperbolic exactly when $\Gamma \le \Mod(S)$ is convex cocompact. We note that the first examples of hyperbolic surface group extensions follow from work of Thurston, whose geometrization theorem for fibered $3$--manifolds produces examples of hyperbolic surface-by-cyclic groups. Later, Mosher \cite{Mosherhypbyhyp} constructed more general hyperbolic surface-by-free groups using the Bestvina--Feighn combination theorem \cite{BF92}.

Since their introduction, there have been several additional characterizations of convex cocompact subgroups of $\Mod(S)$. A particularly useful characterization of convex cocompactness is the following theorem of Kent--Leininger and Hamenst\"{a}dt. In the statement, $\C(S)$ denotes  the curve complex for the closed surface $S$. 

\begin{theorem}[Kent--Leininger \cite{KentLein}, Hamenst\"adt \cite{H}]
\label{thm:mod_CC}
A finitely generated subgroup $\Gamma \le \Mod(S)$ is convex cocompact if and only if some (any) orbit map $\Gamma \to \C(S)$ is a quasi-isometric embedding.
\end{theorem}

From this we see that the surface group extension $E_\Gamma$ is hyperbolic if the orbit map from $\Gamma\le \Mod(S)$ into the curve complex is a quasi-isometric embedding. Hence, strong geometric features of surface group extensions arise from combinatorial conditions on their corresponding subgroups of $\Mod(S)$. With \Cref{th: intro_main_1}, we provide analogous conditions under which combinatorial information about a subgroup $\Gamma\le\Out(\free)$ implies geometric information about the corresponding free group extension $E_\Gamma$.

\begin{remark} 
The condition that every infinite order element of $\Gamma$ is atoroidal is necessary for $E_{\Gamma}$ to be hyperbolic, but this condition is not implied by having a quasi-isometric orbit map into the factor complex $\fc$. This contrasts the surface group situation (c.f., \Cref{thm:mod_CC}), where having a quasi-isometric orbit map $\Gamma\to\C(S)$ automatically implies every infinite order element of $\Gamma$ is pseudo-Anosov. Indeed, there are elements of $\Out(\free)$ that act with positive translation length on $\F$ but are not atoroidal. By Bestvina--Handel \cite{BH92}, these all arise as pseudo-Anosov mapping classes on surfaces with a single puncture. Since such outer automorphisms each fix a conjugacy class in $\free$ (corresponding to the loop enclosing the puncture), they cannot be contained in a subgroup $\Gamma$ for which $E_{\Gamma}$ is hyperbolic. 
\end{remark}

We conclude this section with a brief review of previous examples of hyperbolic extensions of free groups. In \cite{BF92}, Bestvina and Feighn produce examples of hyperbolic free-by-cyclic groups (i.e. $\Gamma \cong \Z$) using automorphisms assumed to satisfy the Bestvina--Feighn flaring conditions. Later, Brinkmann showed that any atoroidal automorphism induces a hyperbolic free-by-cyclic group by showing that all such automorphisms satisfy these flaring conditions \cite{Brink}. This is recorded in \Cref{brink} below. 

The first examples where $\Gamma \le \Out(\free)$ is not cyclic are given in \cite{BFHlam}. There, Bestvina, Feighn, and Handel show that if one starts with fully irreducible and atoroidal elements $\phi,\psi \in \Out(\free)$ that do not have a common power, then there is an $N \ge 1$ such that $\Gamma = \langle \phi^N ,\psi^N \rangle$ is a rank $2$ free group and the corresponding extension $E_{\Gamma}$ is hyperbolic. A different proof of this fact (still using the Bestvina--Feighn combination theorem) is given by Kapovich and Lustig, who additionally show that for large $N$ each nonidentity element of $\Gamma$ is fully irreducible \cite{KLping}.

\subsection{Outline of proof}
To show that the extension $E_{\Gamma}$ is hyperbolic, we use the combination theorem of Mj--Sardar \cite{MjSardar}, which is recalled in \Cref{sec:metric_bundles}. Their theorem states that if a metric graph bundle satisfies a certain \emph{flaring property} (terminology coming from the Bestvina--Feighn combination theorem), then the bundle is hyperbolic. Using the map between the Cayley graphs of $E_{\Gamma}$ and $\Gamma$ as our graph bundle, we show in \Cref{sec: hyp_bundle} that this flaring property is implied by the following conjugacy flaring property of $\Gamma \le \Out(\free)$. First let $S$ be a finite symmetric generating set for $\Gamma$ with associated word norm $\abs{\cdot}_S$. Also fix a basis $X$ for $\free$. We say that $\Gamma$ has \define{$(\lambda ,M)$--conjugacy flaring} for the given $\lambda > 1$ and positive integer $M\in \N$ if the following condition is satisfied:
\begin{itemize}
\item[] For all $\alpha \in \free$ and $g_1,g_2 \in \Gamma$ with $\abs{g_i}_S \ge M$ and $\abs{g_1g_2}_S = \abs{g_1}_S +\abs{g_2}_S$, we have
\[\lambda \norm{\alpha}_X \le \max \left\{\norm{g_1(\alpha)}_X, \norm{g_2\inv(\alpha)}_X\right\},\]
where $\norm{\cdot}_X$ denotes conjugacy length (i.e., the shortest word length with respect to $X$ of any element in the given conjugacy class).
\end{itemize}
\Cref{prop: bundle_flaring} shows that if $\Gamma \le \Out(\free)$ has conjugacy flaring, then $E_\Gamma$ has the Mj--Sardar flaring property and, hence, $E_{\Gamma}$ is hyperbolic.
Thus it suffices to show that any $\Gamma\le \Out(\free)$ satisfying the hypotheses of \Cref{th: intro_main_1} has conjugacy flaring. This is accomplished by using the geometry of Outer space.

First, \Cref{prop: stab_intro} is used to show that geodesic words in $(\Gamma, \abs{\cdot}_S)$ are sent via the orbit map $\Gamma \to \X$ to quasigeodesics that fellow travel a special class of paths in $\X$, called \define{folding paths}. Therefore, by the definition of distance in $\os$ (\Cref{pro: distance}), the conjugacy length of $\alpha\in \free$ along the quasigeodesic in $\Gamma$ is proportional to the conjugacy length of $\alpha$ along the nearby folding path. Thus it suffices to show that the length of every conjugacy class ``flares'' along any folding path that remains close to the orbit of $\Gamma$ in $\X$, meaning that the length grows at a definite exponential rate in either the forwards or backwards (see \Cref{sec: qcx_implies_flar} for details.) \Cref{flaring_along_folding} proves exactly this type of flaring  for folding paths that remain close to the orbit of any group $\Gamma$ that satisfies the hypotheses of \Cref{th: intro_main_1}.

To summarize: If the orbit map $\Gamma\to \fc$ is a quasi-isometric embedding and every infinite order element of $\Gamma$ is atoroidal then folding paths between points in the orbit $\Gamma \cdot R$ (for $R \in \X$) have the flaring property (\Cref{sec: qcx_implies_flar}). This, together with the fact that these folding paths fellow travel the image of geodesics in the group $\Gamma$ (\Cref{prop: geos_fellow_travel}), implies that $\Gamma$ has conjugacy flaring (\Cref{qcx_implies_flaring}). Finally, \Cref{prop: bundle_flaring} shows that conjugacy flaring of $\Gamma$ implies that the hypothesis of the Mj--Sardar theorem are satisfied and that $\Gamma$ is hyperbolic.

\subsection{Acknowledgments}
We would like to thank Chris Leininger, Alan Reid, Patrick Reynolds, and Caglar Uyanik for helpful conversations, as well as Ilya Kapovich for valuable comments on an earlier versions of this paper. 
We also thank the anonymous referee for a careful reading of the paper and for several helpful suggestions.
We are grateful to the University of Illinois at Urbana-Champaign and the University of Texas at Austin for hosting the authors during different stages of this project, and to the GEAR Network for supporting a trip during which this paper was completed.

\section{Preliminaries} \label{sec: prelims}

\subsection{Paths}
Throughout this paper, the notation $\I$ (or sometimes $\mathbf{J}$) will be used to denote a closed, connected interval $\I\subseteq \R$. We write $\I_\pm\in \R\cup \{\pm\infty\}$ for the positive and negative endpoints of $\I$, respectively, and correspondingly write $\I = [\Imin,\Ipl]$. By a \define{discrete interval}, we simply mean the integer points $\I\cap \Z$ of an interval $\I\subset \R$. 

A \define{path} in a topological space $Y$ is a map $\gamma\colon \I\to Y$. If $Y$ is a metric space, then the path $\gamma$ is said to be a \define{geodesic} if $d_Y(\gamma(a),\gamma(b)) = \abs{a-b}$ for all $a,b\in \I$ (that is, if $\gamma$ is an isometric embedding of $\I$ into $Y$). A \define{discrete geodesic} is similarly a map $\gamma\colon (\I\cap \Z) \to Y$ of a discrete interval into $Y$ so that $d_Y(\gamma(a),\gamma(b))=\abs{a-b}$ for all $a,b\in \I\cap \Z$. The space $Y$ is a said to be a \define{geodesic metric space} if it is a metric space and for any points $y_+,y_-\in Y$ there exists a finite geodesic $\gamma\colon \I\to Y$ with $\gamma(\I_\pm) = y_\pm$. 

\subsection{Coarse geometry}
Suppose that $X$ and $Y$ are metric spaces. Given a constant $K\geq 1$, a map $f\colon X\to Y$ is said to be a \define{$K$--quasi-isometric embedding} if for all $a,b\in X$ we have
\begin{equation*}\label{eqn:QI}
\frac{1}{K} d_X(a,b) - K \leq d_Y(f(a),f(b)) \leq K d_X(a,b) + K.
\end{equation*}
More generally, the map is said to be \define{coarsely $K$--Lipschitz} if the rightmost inequality above holds. A \define{$K$--quasi-isometry} is a $K$--quasi-isometric embedding $f\colon X\to Y$ whose image $f(X)$ is $D$--dense for some $D\ge 0$. (This the equivalent to the existence of a $K'$--quasi-isometric embedding $g\colon Y\to X$ for which $f\circ g$ and $g\circ f$ are within bounded distance of $\operatorname{Id}_Y$ and $\operatorname{Id}_X$, respectively.)

A \define{$K$--quasigeodesic} in a metric space $Y$ is a $K$--quasi-isometric embedding $\gamma\colon \I\to Y$ of an interval $\I\subset \R$ into $Y$. Similarly, a \define{discrete $K$--quasigeodesic} is a $K$--quasi-isometric embedding $\gamma\colon (\I\cap \Z)\to Y$ of a discrete interval into $Y$.

For $A\ge 0$, the \define{$A$--neighborhood} of a subset $Z$ of a metric space $Y$ will be denoted
\[\nbhd{A}{Z} \colonequals \big\{y\in Y \mid \inf \{d(z,y) \mid z\in Z\}< A \big\}.\]
The \define{Hausdorff distance} between two subsets $Z,Z'\subset Y$ is then defined to be
\[\dhaus(Z,Z') \colonequals \inf \big\{\epsilon > 0 \mid Z\subset \nbhd{\epsilon}{Z'}\text{ and }Z'\subset \nbhd{\epsilon}{Z}\big\}.\]
Finally, when $Y$ is a geodesic metric space, a subset $Z\subset Y$ is said to be \define{$A$--quasiconvex} if every (finite) geodesic with endpoints in $Z$ is contained $\nbhd{A}{Z}$.

\subsection{Gromov hyperbolicity}
Given $\delta \ge 0$, a geodesic metric space $Y$ is \define{$\delta$--hyperbolic} if every geodesic triangle $\triangle$ in $Y$ is \define{$\delta$--thin}, meaning that each side of $\triangle$ lies in the $\delta$--neighborhood of the union of the other two sides. A metric space is \define{hyperbolic} if it is $\delta$--hyperbolic for some $\delta\geq 0$. It is a fact (see \cite[Proposition III.H.1.17]{BH}) that if $X$ is a $\delta$--hyperbolic space then there is a constant $\delta' = \delta'(\delta)$ such that every triangle $\triangle$ has a \define{$\delta'$--barycenter}, meaning a point $c\in X$ that lies within $\delta'$ of each side of $\triangle$.

Every hyperbolic metric space $Y$ has a \define{Gromov boundary} $\partial Y$ defined to be the set of equivalence classes of admissible sequences in $Y$, where a sequence $\{y_n\}$ is admissible if if $\lim_{n,m}(y_n\vert y_m)_x = \infty$ and two sequences $\{y_n\},\{z_n\}$ are equivalent if $\lim_{n,m}(y_n\vert z_m)_x = \infty$ for some basepoint $x\in Y$ (here $(a\vert b)_x$ denotes the \define{Gromov product} $(d(a,x)+d(b,x)-d(a,b))/2$ of $a,b\in Y$ with respect to $x\in Y$). 
Notice that by the triangle inequality, the notions of ``admissible'' and ``equivalent'' do not depend on the point $x \in Y$.
One says that the admissible sequence $y_1,y_2,\dotsc\in Y$ \define{converges} to the point $\{y_n\}\in \partial Y$. 
In particular, every quasigeodesic ray $\gamma\colon [0,\infty)\to Y$ converges to a well-defined \define{endpoint at infinity} $\gamma(\infty)\colonequals \{\gamma(n)\}_{n=1}^\infty\in \partial Y$, and we note that any two quasigeodesic rays whose images have finite Hausdorff distance converge to the same endpoint at infinity. We refer the reader to \cite[Section 2.2]{buyalo2007elements} for additional details.

Consequently, to each quasigeodesic $\gamma\colon \I\to Y$ we may associate two well-defined \define{endpoints} $\gamma(\Ipl),\gamma(\Imin)\in Y\cup \partial Y$, where $\gamma(\I_\pm)$ is understood to be a point of $\partial Y$ when $\I_\pm=\pm\infty$ and is a point of $Y$ when $\I_\pm\in \R$. With this terminology, we have the following well-known consequence of hyperbolicity; see \cite[Theorem III.H.1.7]{BH} for a proof.

\begin{proposition}[Stability of quasigeodesics]
\label{prop:general_stability_for_quasis}
For any given $K,\delta > 0$, there exists a \define{stability constant} $R_0 = R_0(\delta,K) > 0$ with the following property: Let $Y$ be a $\delta$--hyperbolic space. If $\gamma\colon \I\to Y$ and $\rho\colon \J\to Y$ are $K$--quasigeodesics with the same endpoints, then $\gamma(\I)$ and $\gamma'(\I')$ have Hausdorff distance at most $R_0$ from each other.
\end{proposition}

Thinness of triangles in a hyperbolic spaces extends to ideal triangles. That is, given $\delta\ge 0$ there is a constant $\delta''$ such that every geodesic triangle with vertices in $X\cup \partial X$ is $\delta''$--thin, and there exists a barycenter point $c\in X$ that lies within $\delta''$ of each side of the triangle \cite[Theorem 6.24]{Jussi}.  

\paragraph{Hyperbolic groups.} Let $\Gamma$ be a finitely generated group. For any finite generating set $S$, we may build the corresponding Cayley graph $\cay{S}{\Gamma}$ and equip it with the path metric in  which all edges have length one. The group $\Gamma$ is then given the subspace metric, which is equal to the word metric for the given generating set $S$. Up to quasi-isometry, this metric is independent of the choice of generating set. Since the inclusion $\Gamma\hookrightarrow \cay{S}{\Gamma}$ is a $1$--quasi-isometry, we often blur the distinction between $\Gamma$ and its Cayley graph when considering $\Gamma$ as a metric space.
Accordingly, the group $\Gamma$ is said to be \define{$\delta$--hyperbolic} if there is a finite generating set whose Cayley graph is $\delta$--hyperbolic. In this case, boundary $\partial \Gamma$ of $\Gamma$ is defined to be the Gromov boundary of the Cayley graph. Equivalently $\partial \Gamma$ is the set of equivalence classes of discrete quasigeodesic rays $\gamma\colon \N\to \Gamma$. 

\subsection{Metric bundles}
\label{sec:metric_bundles}
We will make use of the concept of metric graph bundles introduced by Mj and Sardar in \cite{MjSardar}. Let $X$ and $B$ be connected graphs equipped their respective path metrics (in which each edge has length $1$), and let $p\colon X\to B$ be a simplicial surjection. Write $V(B)$ for the vertex set of the graph $B$. We say that $X$ is a \define{metric graph bundle over $B$} if there is a function $f\colon \N\to \N$ so that
\begin{itemize}
\item For each vertex $b\in V(B)$, the fiber $F_b = p\inv(b)$ is a connected subgraph of $X$ and the induced path metric $d_b$ on $F_b$ satisfies $d_b(x,y) \le f(d_X(x,y))$ for all vertices $x,y$ of $F_b$.
\item For any adjacent vertices $b_1,b_2\in V(B)$ and any vertex $x_1\in F_{b_1}$, there is a vertex $x_2\in F_{b_2}$ that is adjacent to $x_1$.
\end{itemize} 

Suppose now that $p\colon X\to B$ is a metric graph bundle. By a \define{$k$--qi lift} of a geodesic $\gamma\colon \I\to B$ (where $k\ge 1$) we mean any $k$--quasigeodesic $\tilde{\gamma}\colon \I\to X$ such that $p(\tilde{\gamma}(n)) = \gamma(n)$ for all $n\in \I\cap \Z$. We then say that the metric bundle $p\colon X\to B$ satisfies the \define{flaring condition} if for all $k \ge 1$ there exists $\lambda_k > 1$ and $n_k,M_k\in \N$ such that the following holds: For any geodesic $\gamma\colon[-n_k,n_k]\to B$ and any two $k$--qi lifts $\tilde{\gamma}_1$ and $\tilde{\gamma}_2$ satisfying $d_{\gamma(0)}(\tilde{\gamma}_1(0),\tilde{\gamma}_2(0)) \ge M_k$ we have
\[\lambda_k \cdot d_{\gamma(0)}(\tilde{\gamma}_1(0),\tilde{\gamma}_2(0)) \le \max\{d_{\gamma(n_k)}(\tilde{\gamma}_1(n_k),\tilde{\gamma}_2(n_k)), d_{\gamma(-n_k)}(\tilde{\gamma}_1(-n_k),\tilde{\gamma}_2(-n_k))\}.\]

The following combination theorem of Mj and Sardar \cite{MjSardar} is the key tool that allows us to prove hyperbolicity of group extensions. It builds on the original Bestvina--Feighn combination theorem \cite{BF92} (in the case where $B$ is a tree) and is also related to a combination theorem of Hamenst\"adt \cite{H}.

\begin{theorem}[Mj--Sardar \cite{MjSardar}]
\label{thm:bundle_hyperbolicity}
Suppose that a metric graph bundle $p\colon X\to B$ satisfies:
\begin{enumerate}
\item $B$ is $\delta$--hyperbolic, and each fiber $F_b = p\inv(b)$, for $b\in V(B)$, is $\delta$--hyperbolic with respect to the path metric $d_b$ induced by $X$,
\item for each $b\in V(B)$, the set of barycenters of ideal triangles in $F_b$ is $D$--dense, and
\item the flaring condition holds.
\end{enumerate}
Then $X$ is a hyperbolic metric space.
\end{theorem}

\subsection{Free group extensions} \label{sec: extensions}
In general, an \define{$\free$--extension} is any group $E$ that fits into a short exact sequence of the form
\begin{align} \label{free_extension}
1 \longrightarrow \free \longrightarrow E \longrightarrow  Q \longrightarrow 1.
\end{align}
We often blur the distinction between the group $E$ and the short exact sequence itself. Every such extension gives rise to a homomorphism $\chi \colon Q \to \Out(\free)$ by sending $q\in Q$ to the outer automorphism class of $(\alpha\mapsto \tilde{q}\alpha\tilde{q}\inv)\in \Aut(\free)$, where $\tilde{q}\in E$ is any lift of $q$.
Since different choices of lift give automorphisms that differ by conjugation by an element of $\free$, this gives a well defined homomorphism to $\Out(\free)$. Conversely, any homomorphism $\chi\colon Q\to \Out(\free)$ gives rise to to a $\free$--extension $E_\chi$ via the fiber product construction:
\[ E_{\chi} \colonequals \{(t,q)\in \Aut(\free) \times Q \mid p(t) = \chi(q)\}. \]
Indeed, if $E$ is the extension in \eqref{free_extension} with corresponding homomorphism $\chi\colon Q\to \Out(\free)$, then $E\cong E_{\chi}$.  In the case of a subgroup $\Gamma\le \Out(\free)$, we write $E_\Gamma$ for the $\free$--extension induced by the inclusion $\Gamma\hookrightarrow \Out(\free)$. 

As in the introduction, there is a canonical short exact sequence
\[1 \longrightarrow \free \overset{i}{\longrightarrow} \Aut(\free) \overset{p}{\longrightarrow} \Out(\free) \longrightarrow 1. \]
This sequence is natural for $\free$--extensions in the sense that any extension $E$ as in \eqref{free_extension} with corresponding homomorphism $\chi\colon Q\to \Out(\free)$ fits into a commutative diagram
\begin{equation*}\label{E:innautout}
		\begin{tikzpicture}[>= to, line width = .075em, 
			baseline=(current bounding box.center)]
		\matrix (m) [matrix of math nodes, column sep=1.5em, row sep = 1.5em, 		text height=1.5ex, text depth=0.25ex]
		{
			1 & \free  & E\cong E_\chi & Q & 1 \\
			1 & \free  & \Aut(\free) & \Out(\free) & 1 \\
		};
		\path[->,font=\scriptsize]
		(m-1-1) edge 					(m-1-2)
		(m-1-2) edge 					(m-1-3)
		(m-1-3) edge 					(m-1-4)
		(m-1-4) edge 					(m-1-5)
		
		(m-2-1) edge 					(m-2-2)
		(m-2-2) edge node[above] {$i$}	(m-2-3)
		(m-2-3) edge node[above] {$p$}	(m-2-4)
		(m-2-4) edge 					(m-2-5)
		
		(m-1-3) edge node[right] {$\hat \chi$} 	(m-2-3)
		(m-1-4) edge node[right] {$\chi$}	(m-2-4)
		;
		\draw[double distance = .15em,font=\scriptsize]
		(m-2-2) -- 					(m-1-2)
		;
		\end{tikzpicture}
	\end{equation*}
in which $\hat\chi$ is the restriction of the projection $\Aut(\free)\times Q\to \Aut(\free)$ to $E_\chi$. Therefore $\hat\chi$ surjects $E_\chi$ onto the the preimage of $\chi(Q)\le \Out(\free)$ in $\Aut(\free)$. From this we note that the $\ker(\hat{\chi}) = 1\times \ker(\chi)\le \Aut(\free)\times Q$; thus $\hat{\chi}$ and $\chi$ have isomorphic kernels.
Moreover, we see that in the case of a subgroup $\Gamma\le\Out(\free)$, the extension $E_\Gamma$ mentioned above is isomorphic to the preimage $E_\Gamma = p\inv(\Gamma)\le \Aut(\free)$.

Note that in order for extension $E_{\chi}$ to be hyperbolic, it is necessary that the map $\chi\colon Q \to \Out(\free)$ have finite kernel and for its image to by purely hyperbolic; in which case the above shows $E_{\chi}$ is quasi-isometric to $E_{\chi(Q)}$.
 Otherwise, it is easily seen that $E_\chi$ contains a $\Z \oplus \Z$ and thus that $E_\chi$ cannot be hyperbolic. Hence, to address the question of hyperbolicity of $\free$--extensions, it suffices to focus on the case of extensions $E_\Gamma$ associated to subgroups $\Gamma \le \Out(\free)$. With this perspective, we only consider such extensions $E_\Gamma$ throughout the rest of this paper.


\subsection{Metric properties of Outer space}
\label{sec:metric_props_of_os}

\paragraph{Outer space.}
Let $\free$ denote the free group of rank $r = \rank(\free)$. Since $\free$ is fixed throughout our discussion, its rank $r$ will often be suppressed from the notation. Letting $\rose$ denote the $r$--petal rose (that is, a wedge of $r$ circles) with vertex $v\in \rose$, we fix once and for all an isomorphism $\free\cong \pi_1(\rose,v)$. A \define{graph} is a $1$--dimensional CW complex, and a connected, simply connected graph is a \define{tree}. A \define{core graph} is a graph all of whose vertices have valance at least $2$. Any connected graph $G$ with nontrivial, finitely generated fundamental group has a unique core subgraph whose inclusion into $G$ is a homotopy equivalence. This subgraph is called the \define{core} of $G$.

Culler and Vogtmann's \cite{CVouter} outer space $\os$ of marked metric graphs will play a central role in our discussion. A \define{marked graph} $(G, g)$ is a core graph $G$ together with a homotopy equivalence $g\colon \rose \to G$, called a marking. A \define{metric} on $G$ is a function $\ell\colon E(G) \to \R_{> 0}$ from the set of edges of $G$ to the positive real numbers; we say that an edge $e \in E(G)$ of $G$ has length $\ell(e)$. The \define{volume} of $G$ is defined to be $\sum_{e\in E(G)} \ell(e)$. 
We view the metric $\ell$ as making $G$ into a path metric space in which each edge $e$ has length $\ell(e)$.
A \define{marked metric graph} is then defined to be the triple $(G,g,\ell)$, and we say that two triples $(G_1,g_1,\ell_1)$ and $(G_2,g_2,\ell_2)$ are \define{equivalent} if there is a graph isometry $\phi\colon G_1 \to G_2$ that preserves the markings in the sense that $\phi \circ g_1$ is homotopic to $g_2$. \define{Outer space} $\X$ is the set of equivalence classes of marked metric graphs of volume $1$. We use the notation $\cv$ to denote \define{unprojectivized outer space}, which is the space of marked metric graphs with no restriction on volume. When discussing points in $\os$ or $\cv$ we typically suppress the marking/metric and just write the core graph.  

\paragraph{Conjugacy classes.}
The marking $\rose\to G$ attached to a point $G\in \os$ allows us to view any nontrivial conjugacy class $\alpha$ in $\free$ as a homotopy class of loops in the core graph $G$.  Following the notation of \cite{BFhyp}, we denote the unique immersed loop in this homotopy class by $\alpha\vert G$, which we view as an immersion of $S^1$ into $G$. 
We use $\ell(\alpha\vert G)$ to denote the \define{length of $\alpha$ in $G\in \os$}, that is, the sum of the lengths of the edges crossed by $\alpha\vert G$, counted with multiplicities. Note that if $X = \{x_1,\dotsc,x_r\}$ is a free basis of $\free$ and $G\in \os$ is the rose whose edges each have length $\nicefrac{1}{r}$ and are consecutively labeled by the elements $x_1,\dotsc,x_r$, then $r \cdot \ell(\alpha\vert G)$ is simply the \define{conjugacy length} $\norm{\alpha}_X$ of $\alpha$ with respect to the free basis $X$. That is, $r \cdot \ell(\alpha\vert G)=\norm{\alpha}_X$ is the length of the shortest word in the letters $x_1^\pm,\dotsc,x_r^\pm$ that represents an element of the conjugacy class $\alpha$. We often blur the distinction between an element of $\free$ and its conjugacy class.

The \define{standard topology on $\os$} is defined to be the coarsest topology such that all of the length functions $\ell(\alpha\vert\cdot\;)\colon \os \to \R_+$ are continuous \cite{CVouter}.  Though we will not discuss it, this topology may also be obtained as a simplicial complex with missing faces, or as the equivariant Gromov-Hausdorff topology (see \cite{CVouter} and \cite{paulin1989gromov}). For $\epsilon > 0$, we additionally define the \define{$\epsilon$--thick part of $\os$} to be the subset
\[\os_\epsilon \colonequals \{G\in \os \mid \ell(\alpha\vert G) \ge \epsilon\text{ for every nontrivial conjugacy class $\alpha$ in $\free$}\}.\]

\paragraph{Lipschitz metric.} 
A \define{difference of markings} from $G\in \os$ to $H\in \os$ is any (not necessarily cellular) map $\phi\colon G\to H$ that is homotopic to $h\circ g\inv$, where $g$ and $h$ are the markings on $G$ and $H$, respectively. The \define{Lipschitz distance} from $G$ to $H$ is then defined to be
\[d_{\X}(G,H) \colonequals \inf \{ \log\left(\mathrm{Lip}(\phi)\right) \mid \phi \simeq h \circ g^{-1}\},\]
where $\mathrm{Lip}(\phi)$ denotes the Lipschitz constant of the difference of markings $\phi$. While $d_\os$ is in general asymmetric (that is, $d_\os(G,H)\neq d_\os(H,G)$), we often regard it as a metric since it satisfies definiteness ($d_\os(G,H) = 0$ iff $G=H$) and the ordered triangle inequality ($d_\os(E,H) \leq d_\os(E,G) + d_\os(G,H)$). Its symmetrization 
\[\dsym(G,H) \colonequals d_\os(G,H) + d_\os(H,G)\]
is therefore an honest metric on $\os$, which we note induces the standard topology \cite{FMout}. The preference to work with the asymmetric metric $d_{\X}$ comes from the fact, discussed below, that folding paths are directed geodesics, whereas the symmetrized metric on $\X$ is not a geodesic metric.

Note that for any $\alpha \in \free$ and any difference of marking $\phi\colon G \to H$, by definition we have $\ell(\alpha\vert H) \leq \mathrm{Lip}(\phi) \cdot \ell(\alpha\vert G)$. This implies that
\begin{eqnarray} \label{ineq}
\log\left(\sup_{\alpha \in \free} \frac{\ell(\alpha\vert H)}{ \ell(\alpha\vert G)}\right) \le \log \left(\inf_{\phi} \; \mathrm{Lip}(\phi)\right) =   d_{\X}(G,H).
\end{eqnarray}
We will see below that this is, in fact, an equality; see also \cite{FMout} and \cite{Bbers}. It follows that for any free basis $X = \{x_1,\dotsc,x_r\}$ of $\free$ and any point $G\in \os$, there is a constant $K = K(X,G)\ge 1$ such that $\frac{1}{K}\norm{\alpha}_X \le \ell(\alpha\vert G) \le K\norm{\alpha}_X$ for every conjugacy class $\alpha$ in $\free$.

\paragraph{Coping with asymmetry.}
Since the Lipschitz metric $d_\os$ is not symmetric, some care must be taken when discussing distances in $\os$. Thankfully, the difficulty is somewhat mitigated in the thick part $\os_\epsilon$.

\begin{lemma}[Handel--Mosher \cite{handel2007expansion}, Algom-Kfir--Bestvina \cite{AlBest}] \label{lem: symmetric_in_thick}
For any $\epsilon>0$, there exists $\sym_\epsilon\geq 1$ so that for all $G,H\in \os_\epsilon$ we have
\[ d_\os(H,G) \le \dsym(H,G) = \dsym(G,H) \le \sym_\epsilon \cdot d_\os(G,H).\]
\end{lemma}

Nevertheless, whenever discussing neighborhoods, we always use the symmetrized distance $\dsym$. That is, the \define{$A$--neighborhood} of a subset $Z\subset \os$ is defined to be
\[\nbhd{A}{Z} \colonequals \big\{G\in \os \mid \inf \{\dsym(G,H) \mid H\in Z\}< A \big\}.\]
In particular, if $G\in \nbhd{A}{Z}$, then there exists some $H\in Z$ so that $d_\os(G,H)$ and $d_\os(H,G)$ are both less than $A$. Note that by \cite{FMout}, if $Z \subset \X$ is compact, then so is the closed neighborhood $\overline{\nbhd{A}{Z} }$. The \define{Hausdorff distance} between two subsets of $\os$ is then defined as usual using these symmetrized neighborhoods. 

We will say that two paths $\gamma\colon \I\to \os$ and $\gamma'\colon \I'\to \os$ have the \define{same terminal endpoint} if either $\Ipl,\Ipl'<\infty$ and $\gamma(\Ipl)=\gamma'(\Ipl')$, or if $\Ipl=\Ipl'=\infty$ and the sets $\gamma([t,\infty))$ and $\gamma'([t',\infty))$ have finite Hausdorff distance for all $t\in \I$ and $t'\in \I'$. Having the \define{same initial endpoint} is defined similarly. Accordingly, $\gamma$ and $\gamma'$ are said to \define{have the same endpoints} if their initial and terminal endpoints agree. 

By a \define{geodesic in $\os$} we always mean a \emph{directed} geodesic, that is, a path $\gamma\colon \I\to \os$ such that $d_\os(\gamma(s),\gamma(t)) = t-s$ for all $s<t$. Similarly a \define{$K$--quasigeodesic in $\os$} means a path $\gamma\colon \I\to \os$ so that
\[\tfrac{1}{K}(t-s) - K \le d_\os(\gamma(s),\gamma(t)) \le K(t-s) + K\]
for all $s<t$. Note that a $K$--quasigeodesic typically will not be a $K$--quasigeodesic when traversed in reverse.

\begin{convention} 
Our default metric on $\X$ is the Lipschitz metric and geodesics are directed geodesics with respect to this metric. When discussing neighborhoods and Hausdorff distance, however, we make use of the symmetrized metric as discussed above.
\end{convention}

\subsection{Navigating outer space}

\paragraph{Optimal maps.}
For any $G,H\in \os$, there exits a (nonunique) difference of markings $\phi\colon G\to H$ that realizes the infimum in the definition of $d_\os(G,H)$ \cite{FMout, Bbers}. Such a map is called \define{optimal}. Here, we describe some structure of optimal maps and refer to the references above for details. Firstly, we say that a difference of markings $\phi\colon G\to H$ is \define{linear on edges} if $\phi$ has a constant slope $\sigma(e)$ on each edge $e$ of $G$, meaning that $\phi$ is a local $\sigma(e)$--homothety on $e$ with respect to the local path metrics on $G$ and $H$. In this case $\mathrm{Lip}(\phi) = \max_e\{\sigma(e)\}$. We define the \define{tension subgraph $\triangle_\phi$} to be the subgraph of $G$ consisting of maximally stretched edges, that is, the edges $e$ of $G$ with $\sigma(e) = \mathrm{Lip}(\phi)$. Since every difference of markings is homotopic rel vertices to a map that is linear on edges and whose Lipschitz constant is no greater than the original, we may always suppose optimal maps are linear on edges.

\paragraph{Train tracks.}
Let us define a \define{segment} $[p,q]$ between points $p,q\in G$ to be a locally isometric immersion $[0,L]\to G$ of an interval $[0,L]\subset \R$ sending $0\mapsto p$ and $L\mapsto q$. A \define{direction} at $p\in G$ is a germ of nondegenerate segments $[p,q]$ with $p\neq q$. A map $\phi\colon G\to H$ that is linear on edges with slope $\sigma(e)\neq 0$ for all edges $e$ of $G$ then induces a derivative map $D_\phi$ which sends a direction at $p$ to a direction at $\phi(p)$. We say that two directions at $p\in G$ are in the same \define{gate} if the directions are identified by $D_\phi$. The gates form an equivalence relation on the set of directions in $G$. 

An unordered pair $\{d,d'\}$ of distinct directions at a vertex $v$ of $G$ is called a \define{turn}. The turn $\{d,d'\}$ is \define{illegal} (with respect to $\phi$) if $d$ and $d'$ belong to the same gate and is \define{legal} otherwise. Accordingly, the set of gates in $G$ is also called the \define{illegal turn structure} on $G$ induced by $\phi$. An illegal turn structure is moreover a \define{train track structure} if there are at least two gates at each $p\in G$. This is equivalent to requiring that $\phi$ is locally injective on (the interior of) each edge of $G$ and that every vertex has at least $2$ gates.

For any $G,H \in \X$ there is an optimal map $\phi\colon G \to H$ such that $\triangle_{\phi}$ is a core graph and the illegal turn structure induced by the restriction of $\phi$ to $\triangle_{\phi}$ is a train track structure \cite{FMout, Bbers}. Hence, the tension subgraph $\triangle_\phi$ contains an immersed loop that is legal (crosses only legal turns). If $\alpha$ denotes the conjugacy class represented by a legal loop contained in $\triangle_{\phi}$, it follows that $\ell(\alpha\vert H) = \Lip(\phi)\cdot\ell(\alpha\vert G)$.  Conversely, any difference of markings $\phi\colon G\to H$ satisfying $\ell(\alpha\vert H)=\Lip(\phi)\cdot \ell(\alpha\vert G)$ for some conjugacy class $\alpha$  is necessarily optimal. The existence of optimal maps thus shows that the inequality in \eqref{ineq} is in fact an equality. We collect these facts into the following proposition:

\begin{proposition}[See Francaviglia--Martino \cite{FMout} or Algom-Kfir \cite{AKaxis}]\label{pro: distance}
For every $G\in \os$ there exists a finite set $\mathcal{C}_G$ of primitive conjugacy classes, called \define{candidates}, whose immersed representatives in $G$ cross each edge at most twice and such that for any $H\in \os$
\[
d_\os(G,H) = \max_{\alpha\in \mathcal{C}_G} \log\left(\frac{\ell(\alpha\vert H)}{ \ell(\alpha\vert G)}\right) = \sup_{\alpha\in \free} \log\left(\frac{\ell(\alpha\vert H)}{ \ell(\alpha\vert G)}\right).\]
\end{proposition}

\paragraph{Folding.}
For a linear difference of markings $\phi\colon G\to H$, if $\triangle_{\phi} = G$ and $\phi$ induces a train track structure on $G$, then $\phi$ induces a unique \define{folding path} $\gamma = \gamma^\phi\colon [0,L]\to \os$ with $\gamma^\phi(0) = G$ and $\gamma^\phi(L) = H$ such that $d_\os(\gamma^\phi(s),\gamma^\phi(t)) = t-s$ for $0\leq s \leq t\leq L$. Thus $\gamma^\phi$ is a (directed) geodesic from $G$ to $H$.  The path $\gamma^\phi$ is obtained by ``folding all illegal turns at unit speed,'' as follows: Fix some sufficiently small $\epsilon > 0$. Then for $0\le s \le \epsilon$, form the quotient graph $\hat{G}_s$ obtained by declaring two points $x,y\in G$ to be equivalent if $\phi(x) = \phi(y)$ and $d(x,v) = d(y,v) \le s$ for some vertex $v$ of $G$. Then $\phi$ factors through the quotient map $G\to \hat{G}_s$, and $\hat{G}_s$ inherits a natural metric so that this quotient map is a local isometry on each edge of $G$. If we let $L_s = \vol(\hat{G}_s)$, then the rescaled graph $\bar{G}_s = (1/L_s)\hat{G}_s$ lies in $\os$ (i.e., has volume $1$), and $\phi$ factors as a composition $G \to \bar{G}_s \to  H$ of two optimal maps with Lipschitz constants $L_s$ and $L/L_s$, respectively. Accordingly, we set $\gamma^\phi(\log(L_s)) = \bar{G}_s$. This defines $\gamma^\phi(t)$ for all sufficiently small $t\ge 0$. Staring now with the optimal map $\bar{G}_\epsilon\to H$, we may repeat this procedure to define $\gamma^\phi(t)$ for more values of $t$. While it is not obvious, after finitely many iterations we will eventually arrive at $\gamma^\phi(L) = H$. See \cite[Proposition 2.2]{BFhyp} for a justification of this claim and a more detailed construction of $\gamma^\phi$.

\begin{remark}
The folding paths used this paper, as defined above, are sometimes called ``greedy folding paths'' \cite{BFhyp} or ``fast folding lines'' \cite{FMout} in the literature. These are a special case of the more flexible ``liberal folding paths'' that are sometimes considered (see the appendix of \cite{BFproj}) and which include the original Stallings paths \cite{StallingsGraphs}.
\end{remark}

If $\gamma^\phi\colon [0,L]\to \os$ is a folding path, as described above, we often use $G_t$, $t\in[0,L]$, to denote $\gamma^\phi(t)$. Observe that for all $0\le s \le t \le L$, the construction of $\gamma^{\phi}$ provides induced optimal maps $\gamma^\phi_{st}\colon G_s\to G_t$, which we refer to as \define{folding maps}. These maps compose so that $\gamma^\phi_{rt} = \gamma^\phi_{st}\circ \gamma^\phi_{rs}$ for $r\le s\le t$, and we additionally have $\gamma^{\phi}_{0L} = \phi$ and $\gamma^\phi_{tt} = \mathrm{Id}_{G_t}$ for all $t$. Furthermore, for all $t > s$, the maps $\gamma^\phi_{st}\colon G_s\to G_t$ (i) induce the same train track structure on $G_s$ (independent of $t$), (ii) send legal segments (segments crossing only legal turns) to legal segments, and (iii) have associated folding paths exactly given by the restrictions $\gamma^{\phi}\vert_{[s,t]}$. 

Lastly, we note that it is also possible to construct \emph{biinfinite} folding paths, by which we mean a directed geodesic $\gamma\colon \R\to \os$ together with with maps $\gamma_{st}\colon G_s\to G_t$ (where $G_t = \gamma(t)$) for all $s\le t$ satisfying the above properties.

\paragraph{Standard geodesics.} 

It is not true that any two points of $G,H\in\os$ may be connected by a folding path. There is, however, a nonunique \define{standard geodesic} from $G$ to $H$ \cite{FMout}. In \cite[Proposition 2.5]{BFhyp}, Bestvina and Feighn give a detailed construction of such a standard geodesic, which we summarize here: First, take an optimal map $\phi\colon G\to H$ that is linear on edges and consider the tension subgraph $\triangle_\phi$ of $G$. Let $\Sigma_G\subset \os$ denote the simplex of all (volume--$1$) length functions on the marked graph $G$. By shortening some of the edges outside of $\triangle_{\phi}$ (and rescaling to maintain volume $1$), one may then find a point $G'\in \overline{\Sigma_G}$ in the closed simplex together with an optimal difference of markings $\phi'\colon G'\to H$ whose tension graph $\triangle_{\phi'}$ is all of $G'$ and such that
\[d_\os(G,H) = d_\os(G,G') + d_\os(G',H).\]
If $\gamma_1$ denotes the linear path in $\overline{\Sigma_G}$ from $G$ to $G'$ (which when parameterized by arc length is a directed geodesic) and $\gamma_2 = \gamma^{\phi'}$ denotes the folding path from $G'$ to $H$ induced by $\phi'$, it follows from the equation above that the concatenation $\gamma_1\gamma_2$ is a directed geodesic from $G$ to $H$.

Let us introduce the following terminology. By a \define{rescaling path} we mean a linear path $\I\to \os$ in a closed simplex $\overline{\Sigma_G}$ parameterized by arclength. While such a path can in principle have infinite length in the negative direction (if the volume of a core subgraph tends to $0$ as $t\to -\infty$), every rescaling path has finite length in the forward direction since a subgraph can only stretch until its volume is equal to $1$. More specifically:

\begin{lemma}\label{lem:bound_on_scaling_length}
If $G_t$, $t\in [0,L]$, is a rescaling path with $G_0\in \os_\epsilon$, then $L\le \log(2/\epsilon)$.
\end{lemma}
\begin{proof}
Let $\alpha$ be any candidate for $G_0$, so the immersed loop representing $\alpha$ in $G_0$ crosses each edge at most twice. Since $G_0$ and $G_L$ represent the same marked graphs up to collapsing some edges of $G_0$, $\alpha\vert G_L$ crosses no edge more than twice. Thus we have $\ell(\alpha\vert G_L) \le 2$. On the other hand $\ell(\alpha\vert G_0)\ge \epsilon$ by assumption. Thus $\frac{\ell(\alpha\vert G_L)}{\ell(\alpha\vert G_0)} \le \frac{2}{\epsilon}$. Since this holds for every candidate of $G$,  \Cref{pro: distance} implies that $L = d_\os(G_0,G_L) \le \log(2/\epsilon)$, as claimed.
\end{proof}

In general, by a \define{standard geodesic} we mean a (directed) geodesic $\gamma\colon \I\to \os$ that is either a folding path, a rescaling path, or a concatenation $\gamma\colon \I\to \os$ of a rescaling path $\gamma\colon \I^s\to \os$ and a folding path $\gamma\colon \I^f\to \os$, where in the latter case we require $\Ipl^s = \Imin^f \in \R$, that $\I = \I^s\cup \I^f$, and that the concatenation is a directed geodesic. In this latter case the \define{folding image} of the standard geodesic is denoted $\foldim(\gamma) = \gamma(\I^f)$, and the \define{scaling image} is similarly denoted $\scaleim(\gamma) = \gamma(\I^s)$. For notational convenience, when the standard geodesic $\gamma\colon \I\to \os$ is simply a rescaling path, we define $\scaleim(\gamma) = \gamma(\I)$ and $\foldim(\gamma) = \gamma(\Ipl)$ (recall that $\Ipl < \infty$ for rescaling paths); when $\gamma$ is simply a folding path we define $\foldim(\gamma) = \gamma(\I)$ and either $\scaleim(\gamma) = \gamma(\Imin)$ or $\scaleim(\gamma) = \emptyset$ depending on whether $\Imin \in \R$ or $\Imin=-\infty$. In particular, note that the $\foldim(\gamma)$ is nonempty for every standard geodesic.

\paragraph{Folding and unfolding.} In Section $5$ of \cite{BFhyp}, Bestvina and Feighn give a detailed account of what happens to an immersed path in the graph $G_t$ under folding and unfolding. We review the basics here, as they will be needed in \Cref{sec: qcx_implies_flar}. For additional details and examples, see \cite{BFhyp}.

Fix a folding path $\gamma(t) = G_t$ with $t\in [a,b]$, and let $p_b$ be an immersed path in $G_b$. It is always possible to lift (or unfold) $p_b$ to an immersed path $p_t$ in $G_t$ with the property that $p_t$ maps to a path in $G_b$ whose immersed representative, rel endpoints, is $p_b$ (recall that the folding path $\gamma$ comes equipped with folding maps $\gamma_{tb}\colon G_t \to G_b$). These lifts are not necessarily unique, but Bestvina and Feighn show that we can remove segments from the ends of $p_b$ to obtain unique lifts. This is their unfolding principle, which we state as the following lemma:

\begin{lemma}[Unfolding principle \cite{BFhyp}]
\label{lem:unfolding}
With the set up above, lifting $p_b$ to $G_t$ is unique between the first and last illegal turns of $p_b$, including the germs of directions beyond these turns.
\end{lemma}

The process of lifting (uniquely) an immersed segment $p_b$ whose endpoints are at illegal turns is called \define{unfolding}. Note that the unfolding principle applies to an illegal turn itself. In particular, if $\alpha$ is a conjugacy class of $\free$ and $p_b$ is either a subpath of $\alpha\vert G_t$ with endpoints illegal turns or an illegal turn of $\alpha|G_b$,  then $p_b$ unfolds to path (or an illegal turn) $p_t$ that is contained in $\alpha|G_t$. Moreover, multiple occurrences of $p_b$ in $\alpha|G_b$ all unfold to $p_t$ as a subpath of $G_t$. This all follows from the unfolding principle.

Similarly, we can understand the image of certain subpaths $p_a$ of $\alpha|G_a$ under the folding path $G_t$. Note that the image of $p_a$ in $G_t$ is not necessarily contained in the image of $\alpha|G_t$, even after tightening (i.e. passing to the immersed representative). However, if there is a subpath $p_b$ of $\alpha|G_b$ with endpoints at illegal turns which unfolds to $p_a$ in $\alpha|G_a$, then unfolding gives a unique path $p_t$ of $G_t$ whose endpoints are at illegal turns of $G_t$. By the above paragraph, these unfolded paths have the property that $p_t$ is a subsegment of $\alpha|G_t$ for all $t \in [a,b]$.

\paragraph{Projecting to standard geodesics.}
In \cite[Definition 6.3]{BFhyp} Bestvina and Feighn define for any folding path $\gamma\colon \I\to \os$ a \define{projection} $\proj_\gamma\colon \os\to \gamma(\I)$ onto the image of the folding path (one could alternately think of the projection as landing in the domain interval $\I$). As the definition of $\proj_\gamma(H)$ is rather technical---in short it involves looking at the infimum of times $t$ for which a certain cover of $\gamma(t)$ contains an immersed legal segment of length $3$---we delay a careful discussion until \S\ref{sec:BF_projs_morse} where a precise construction of the projection $\proj_\gamma\colon \os\to \gamma(\I)$ is given in \Cref{def:BestvinaFeighn_projection}. However, although $\proj_\gamma$ does coarsely agree with the closest-point-projection to $\gamma(\I)$ in special circumstances (see \Cref{lem:coarsely closest point}), we caution that $\proj_\gamma$ is generally \emph{unrelated} to the closest-point-projection onto $\gamma(\I)$.

Taking the existence of this projection for granted for the time being, we presently extend this construction in the natural way to any standard geodesic $\gamma\colon \I \to \os$ by declaring $\proj_{\gamma}\colonequals \proj_{\gamma^f}\colon \os\to \foldim(\gamma)$, where $\I = \I^s\cup \I^f$ and $\gamma^f = \gamma\vert_{\I^f}$ is the folding portion of $\gamma$. (Recall that $\foldim(\gamma)\neq \emptyset$ for every standard geodesic $\gamma$).

\subsection{The free factor complex}\label{sec: factor_complex}
The \define{(free) factor complex} $\fc$ of $\free$ is the simplicial complex whose vertices are conjugacy classes of nontrivial, proper free factors of $\free$. A collection of vertices $\{[A_0], \ldots, [A_k]\}$ determines a $k$--simplex if, after reordering and choosing conjugacy representatives, we have  $A_0 < \dotsb <  A_k$. The free factor complex was first introduce by Hatcher and Vogtmann in \cite{HVff}. When it should cause no confusing to do so, we will usually drop the conjugacy symbol from the notation and denote a conjugacy class of free factors by $A \in \F^0$. 

We equip the factor complex $\fc$ with its simplicial path metric. That is, we geometrically view $\fc$ as the simplicial graph $\fc^1$ equipped with the path metric in which each edge has length $1$. For our purposes, the significance of the factor complex stems from the following foundational result of Bestvina and Feighn:

\begin{theorem}[Bestvina--Feighn \cite{BFhyp}]
\label{T:hyperbolicity_of_fc}
The factor complex $\fc$ is Gromov-hyperbolic.
\end{theorem}

There is a natural (coarse) projection $\pi\colon \os\to\fc$ defined by sending  $G\in \os$ to the set of free factors corresponding to proper subgraphs of $G$. That is,
\[\pi(G) = \{\pi_1(G'): \text{ $G'$ is a proper, connected, noncontractible subgraph of $G$ }\} \subset \F^0,\]
where $\pi_1(G')\le \pi_1(G)$ is identified with a free factor of $\free$ via the marking $\rose \to G$. This projection is a key tool in the proof of \Cref{T:hyperbolicity_of_fc} above. For $G \in X$, it is quickly verified that $\mathrm{diam}_\fc(\pi(G)) \le 4$ \cite[Lemma 3.1]{BFhyp}. 

Let us define the \define{factor distance} between two points $G,H\in \os$ to be
\[d_\fc(G,H) = \diam_\fc(\pi(G) \cup \pi(H)).\]
Corollary 3.5 of \cite{BFhyp} shows that $d_\fc(G,H) \le 12 e^{d_\os(G,H)} + 32$. In fact, as indicated in \cite{BFhyp}, this may easily be strengthened to show that $\pi$ is coarsely $80$--Lipschitz: 

\begin{lemma}\label{lem: 80Lip}
For any $G,H\in \os$ we have $d_\fc(G,H) \le 80d_\os(G,H) + 80$.
\end{lemma}
\begin{proof}
Let $L = d_\os(G,H)$, and let $\gamma\colon [0,L]\to \os$ be a standard geodesic from $G$ to $H$. If $k = \ceil{L}$, then we may find times $0 = t_0 < \dotsb < t_k = L$ so that $d_\os(\gamma(t_i),\gamma(t_{i+1})) \le 1$ for all $0 \le i < k$. By Corollary 3.5 of \cite{BFhyp}, it follows that $d_\fc(\gamma(t_i),\gamma(t_{i+1})) \le 12e + 32$ for each $i$, and thus that $d_\fc(G,H) \le (12e + 32)k \le 80L + 80$ by the triangle inequality. 
\end{proof}

Similarly, we will use the following easy lemma:

\begin{lemma}\label{lem:easy}
Suppose that there is a nontrivial conjugacy class $\alpha$ which has length less than $1$ on both $G,H  \in \os$. Then $d_\fc(G,H) \le 10$.
\end{lemma}

\begin{proof}
Such a conjugacy class would determine an immersed loop contained in a proper core subgraph of each graph. Hence, $\alpha$ is simultaneously contained in free factors $A$ and $B$ appearing in the diameter-$4$ projections of $G$ and $H$, respectively. In this case, $d_\fc(A,B)\le 2$ (\cite[Section 3.2]{taylor_RAAG1}), showing that the union $\fproj(G)\cup \fproj(H)$ has diameter at most $10$.
\end{proof}

In the process of showing that $\fc$ is hyperbolic, Bestvina and Feighn also prove the following very useful result; it essentially says that the projection onto a folding path is strongly contracting when viewed from the factor complex. 

\begin{proposition}[Bestvina--Feighn {\cite[Proposition 7.2]{BFhyp}}]
\label{prop:BF_folding_paths_contract}
There exists a universal constant ${\sf B}$ (depending only on $\rank(\free)$) such that the following holds. If $H, H' \in \os$ satisfy $d_{\os}(H,H') \le M$ and $\gamma\colon \I\to \os$ is a standard geodesic with $d_{\os}(H,\gamma(t)) \ge M$ for all $t$, then $d_\fc(\proj_{\gamma}(H),\proj_{\gamma}(H'))\leq {\sf B}$. 
\end{proposition}

\begin{remark}
While Proposition 7.2 of \cite{BFhyp} is only stated for the projection to a finite length folding path, it clearly holds for our generalized projection to a finite length standard geodesic. By considering an exhaustion by finite length subpaths, the result is also seen to hold for infinite length standard geodesics.
\end{remark}

\subsection{$\Out(\free)$ basics} \label{sec: Out_basics}
We recall some of the structure of automorphisms of $\free$ and the dynamics of their actions on $\X$ and $\F$. The group $\Out(\free)$ acts naturally on $\os$ by changing the marking: $\phi\cdot (G,g,\ell) = (G,g\circ \hat{\phi}\inv,\ell)$, where $\hat{\phi}\inv\colon \rose\to \rose$ is any homotopy equivalence whose induced map on $\free \cong \pi_1(\rose)$ is in the outer automorphism class $\phi\inv\in \Out(\free)$. One may easily verify that $G\mapsto \phi\cdot G$ defines an isometry of $(\os,d_\os)$. Each outer automorphism $\phi\in\Out(\free)$ permutes the set $\fc^0$ of conjugacy classes of free factors via $\phi\cdot[A] = [\phi(A)]$, and this extends to a simplicial (and hence isometric) action of $\Out(\free)$ on $\fc$. The actions of $\Out(\free)$ on $\os$ and $\fc$ are equivariant with respect to the projection $\pi\colon \os\to \fc$: For each $G \in \X$,
\[\pi(\phi \cdot G) = \phi \cdot \pi(G) \]
as subsets of $\F$.

\paragraph{Full irreducibility.}
We are primarily interested in elements $\phi\in \Out(\free)$ that are \define{fully irreducible}, meaning that no positive power of $\phi$ fixes the conjugacy class of any free factor of $\free$. Hence $\phi$ is fully irreducible if and only if its action on $\F$ has no periodic vertices. In fact, Bestvina and Feighn have shown the following:

\begin{theorem}[Bestvina--Feighn \cite{BFhyp}]
An element $\phi \in \Out(\free)$ acts with positive translation length on the free factor complex $\F$ if and only if $\phi$ is fully irreducible.
\end{theorem}
Recall that the \define{(stable) translation length} of $\phi \in \Out(\free)$ acting on $\F$ is by definition
\[\ell_{\F}(\phi) = \lim _{n \to \infty} \frac{d_{\F}(A, \phi^n A)}{n}, \]
for any $A \in \F^0$. It is well known (and easily verified) that $\ell_{\F}(\phi)$ does not depend on the choice of $A$ and that $\ell_{\F}(\phi^n)= n \cdot \ell_{\F}(\phi)$. Having positive translation length implies that for any $A \in \F^0$, the orbit map $\Z\to \fc$ defined by $n \mapsto \phi^n\cdot A$ is a quasi-geodesic in $\F$. In \Cref{sec: apps} we also discuss translation lengths of elements of $\Out(\free)$ acting on a different hyperbolic complex. Regardless of the context, we call an isometry of a hyperbolic space \define{loxodromic} if it acts with positive translation length.

\paragraph{Hyperbolicity.}
An element $\phi\in \Out(\free)$ is said to be \define{hyperbolic} or \define{atoroidal} if $\phi^i(\alpha)\neq \alpha$ for every nontrivial conjugacy class $\alpha$ in $\free$ and every $i \ge1$. While neither hyperbolicity nor full irreducibility implies the other, there are many automorphisms of $\free$ that have both these properties. Hyperbolic elements of $\Out(\free)$ are essential to our discussion because of the following theorem of Brinkmann.

\begin{theorem}[Brinkmann \cite{Brink}] \label{brink}The outer automorphism class of $\Phi\in \Aut(\free)$ is hyperbolic if and only if the semidirect product $\free\rtimes_\Phi \Z$ is a Gromov-hyperbolic group.
\end{theorem}

We say that $\Gamma\le\Out(\free)$ is \define{purely hyperbolic} if every infinite order element of $\Gamma$ is hyperbolic. Before concluding this section, we observe that when $\Gamma$ is purely hyperbolic there is a uniform upper bound (depending only on $\rank(\free)$) on the number of elements of $\Gamma$ that fix any given conjugacy class. To this end, for $\alpha$ a conjugacy class in $\free$ set 
\[\Gamma_{\alpha} = \{ \phi \in \Gamma : \phi(\alpha) = \alpha   \}.\]

\begin{lemma}\label{lem: phyp_implies_finiteprobs}
There is a constant $e_r$ depending only on the rank $r = \rank(\free)$ such that for any purely hyperbolic $\Gamma\le \Out(\free)$ we have  $|\Gamma_{\alpha}| \le e_r$ for each nontrivial conjugacy class $\alpha$ of $\free$.
\end{lemma}

\begin{proof}
Since $\Gamma$ is purely hyperbolic, $\Gamma_{\alpha}$ is a torsion subgroup of $\Out(\free)$. It is known that any torsion element survives in the quotient $\Out(\free) \to \mathrm{GL}_r(\Z/3\Z)$ \cite{CVouter} and so $\Gamma_{\alpha}$ injects into $\mathrm{GL}_r(\Z/3\Z)$. Hence, we may take $e_r = |\mathrm{GL}_r(\Z/3\Z)|$. 
\end{proof}

\section{Quasiconvexity and folding paths}
\label{sec:folding_path_quasiconvexity}
For the main results of \Cref{sec: progression} we will need to know that outgoing balls in the Lipschitz metric are quasiconvex with respect to folding paths. This is proven in \Cref{cor:right_balls_quasi-convex} below. We first show in \Cref{prop:lenght-quasi-convexity} that the length of every conjugacy class is quasiconvex along folding paths. 

We begin by recalling some notation from \cite{BFhyp}. For a folding path $G_t$, $t\in \I$, define the \define{illegality} $m(G_{t_0})$ of $G_{t_0}$ at time $t_0$ to be 
\[m(G_{t_0}) = \sum_v \sum_{\Omega_v} (|\Omega_v|-1), \]
where $v$ varies over the vertices of $G_{t}$ and $\Omega_v$ varies over all gates of $G_{t_0}$ at the vertex $v$ (so each $\Omega_v$ is an equivalence class of directions at $v$). Note that if we set $M = 6\rank(\free) -6$, which bounds twice the number of edges of any graph in $\X$, then $1 \le m(G_t) \le M$ for all $t$. We often write $m_t$ for $m(G_t)$ when the folding path is understood. For any conjugacy class $\alpha$, we additionally let $k_t = k(\alpha\vert G_t)$ denote the number of illegal turns in $\alpha \vert G_t$. 

In Corollary 4.5 and Lemma 4.4 of \cite{BFhyp}, Bestvina and Feighn show that the function $t\mapsto \ell(\alpha\vert G_t)$ is piecewise exponential and that its right derivative at time $t_0$ is given by
\[\ell(\alpha\vert G_{t_0}) - 2\frac{k(\alpha\vert G_{t_0})}{m(G_{t_0})}.\]
Using this, they prove the following estimate:

\begin{lemma}[Bestvina--Feighn {\cite[Lemma 4.10]{BFhyp}}] \label{lem: length_v_legality}
Suppose that $G_t$, $t\in [0,L]$, is a folding path and that $\alpha$ is any conjugacy class in $\free$. Then for all $t\in [0,L]$ we have
\[\ell(\alpha|G_t) \le \max \{2k(\alpha\vert G_0), \ell(\alpha|G_L)  \}.\]
\end{lemma}

\begin{proposition}[Quasiconvexity of lengths along folding paths]
\label{prop:lenght-quasi-convexity}
Let $G_t$, $t\in [0,L]$, be a folding path, and let $\alpha$ be any conjugacy class in $\free$. Then
\[\ell(\alpha|G_t) \le 6\rank(\free) \cdot \max \left\{ \ell(\alpha|G_0),\ell(\alpha|G_L )\right\}.\]
\end{proposition}
\begin{proof}
Let $l \in [0,L]$ be the supremum of times for which the piecewise exponential function $t \mapsto \ell(\alpha|G_t)$ is decreasing on $[0,l)$. Hence the right derivative of $\ell(\alpha|G_t)$ at time $l$ is nonnegative.  If $l = L$, then we are done. Otherwise, by the derivative formula above we have $\ell(\alpha|G_l) \ge 2\frac{k_l}{m_l} \ge 2\frac{k_l}{6r}$, where $r = \rank(\free)$. Hence
\[k_l \le 3 r \cdot \ell(\alpha|G_l) \le  3 r \cdot \ell(\alpha|G_0)\]
by the choice of $l$. Applying \Cref{lem: length_v_legality}, we see that for all $t \in [l , L]$, 
\begin{eqnarray*}
\ell(\alpha|G_t) &\le& \max \{2 k_l , \ell(\alpha|G_L) \} \\
& \le & \max\{6r \cdot \ell(\alpha|G_0) , \ell(\alpha|G_L) \}.
\end{eqnarray*}
Since $\ell(\alpha|G_t) \le \ell(\alpha|G_0)$ for all $t \in [0,l]$, this completes the proof.
\end{proof}

\begin{corollary}[Outgoing balls are folding-path-quasiconvex]
\label{cor:right_balls_quasi-convex}
There exists a universal constant ${\sf A}$ (depending only on $\rank(\free)$) such that the following holds. For any $H\in \os$ and $R>0$, if $\gamma \colon [0,L]\to \os$ is a folding path $\gamma(t) = G_t$ with $d_\os(H,G_0),d_\os(H,G_L)\leq R$,
then for all $t\in [0,L]$ we have
\[d_\os(H, G_t) \le R + {\sf A}.\]
\end{corollary}
\begin{proof}
Applying \Cref{prop:lenght-quasi-convexity}, for any time $t\in [0,L]$ we have
\begin{align*}
d_\os(H,G_t) &= \sup_{c\in \free}\log\left(\frac{\ell(c \vert G_t)}{\ell(c\vert H)}\right)\\
&\leq \sup_{c\in \free}\log\left( 6\rank(\free)\cdot \frac{\max\left\{\ell(c\vert G_0),\ell(c\vert G_L)\right\}}{\ell(c\vert H)}\right)\\
&\leq \log(6\rank(\free)) +  \max\left\{\sup_{c\in \free}\log\left(\frac{\ell(c\vert G_0)}{\ell(c\vert H)}\right),\sup_{c\in \free}\log\left(\frac{\ell(c\vert G_L)}{\ell(c\vert H)}\right)\right\}\\
&\leq \log(6\rank(\free)) +  R .\qedhere
\end{align*}
\end{proof}

\section{Stability for $\fc$--progressing quasigeodesics}
\label{sec: progression}

In this section we explore the structure of quasigeodesics in Outer space that project to parameterized quasigeodesics in the factor complex. We show that, as in a hyperbolic space, such quasigeodesics are stable in the sense that they fellow travel any geodesic with the same endpoints. More specifically, we prove the following.

\begin{theorem}[$\fc$--progressing quasigeodesics are stable]
\label{prop: geos_fellow_travel} 
\label{prop: stab_intro}
Let $\gamma\colon \I\to \os$ be a \linebreak $K$--quasigeodesic whose projection $\pi \circ \gamma\colon \I\to \fc$ is also a $K$--quasigeodesic. Then there exist constants $A,\epsilon > 0$ and $K'\ge 1$ depending only on $K$ (and the injectivity radius of the terminal endpoint $\gamma(\Ipl)$ when $\Ipl<\infty$) with the following property: If $\rho\colon \J\to \os$ is any geodesic with the same endpoints as $\gamma$, then
\begin{itemize}
\item[(i)] $\gamma(\I), \rho(\J)\subset \os_{\epsilon}$,
\item[(ii)] $d_{\mathrm{Haus}}(\gamma(\I),\rho(\J)) < A$, and
\item[(iii)] $\pi\circ \rho\colon \J\to \fc$ is a (parameterized) $K'$--quasigeodesic.
\end{itemize} 
\end{theorem}

Our proof relies crucially on the projection $\proj_{\gamma}\colon \os \to \gamma(\I)$ from Outer space to the image of any standard geodesic $\gamma\colon \I\to \os$. As recorded in \Cref{prop:BF_folding_paths_contract}, Bestvina and Feighn prove that this projection is strongly contracting when viewed in the factor complex, and they use this to show that $\F$ is $\delta$--hyperbolic \cite{BFhyp}. The projection $\pi\circ\gamma$ of $\gamma$ to the factor complex is also shown to be a unparameterized $K_f$--quasigeodesics, where $K_f$ depends only on $\rank(\free$) \cite{BFhyp}.  As a quasigeodesic, the nearest point retraction $\bold{n}_{\pi\circ\gamma}\colon \F \to \pi(\gamma(\I))$ onto the image $\pi(\gamma(\I))$ is coarsely $L_0$--Lipschitz for some $L_0$ that depends only on $\delta$ and $K_f$. The next lemma verifies that $\bold{n}_{\pi\circ\gamma} \colon \F \to\pi(\gamma(\I))$ agrees with the Bestvina--Feighn projection $\pi \circ \p_{\gamma}\colon \F \to \pi(\gamma(\I))$ up to uniformly bounded error.

\begin{lemma}
\label{lem:folding_projection_is_closest_pt}
There is a constant $D_1\ge 0$, depending only on $\rank(\free$), such that for any $H\in \os$ and any standard geodesic $\gamma\colon \I\to \os$ we have 
\[d_{\F}(\pi(\p_{\gamma}(H)), \bold{n}_{\pi\circ\gamma}(\fproj(H))) \le D_1.\]
\end{lemma}
\begin{proof} 
To simply notation, set $\hat{C} = \pi(\p_{\gamma}(H))$, $C = \fproj(H)$, and $A' =  \bold{n}_{\pi\circ\gamma}(C)$; both of these points lie on the unparameterized $K_f$--quasigeodesic $\pi(\gamma(\I))$. Now let $\hat\rho$ and $\rho'$ be folding paths whose images in $\F$ joint $C$ to $\hat{C}$ and $A'$, respectively. We are now in the situation of \cite[Proposition 9.1]{BFhyp}, which states that there is a $Q'$ on $\pi(\rho')$ whose distance from $\hat{C}$ is no greater than $B_1$, where $B_1$ is a uniform constant.

Since $\pi(\rho')$ is an unparameterized $K_f$--quasigeodesic, any geodesic $[C,A']$ in $\F$ joining $C$ and $A'$ contains a point $Q_0$ with $d_{\F}(Q',Q_0) \le R_0$, where $R_0 = R_0(\delta,K_f)$ is the constant from \Cref{prop:general_stability_for_quasis}. Hence, $d_{\F}(Q_0, \hat{C}) \le B_1 + R_0$. Since, no factor on $\pi(\gamma)$ is closer to $C$ than $A'$, we must have $d_{\F}(Q_0, A') \le B_1 + R_0$. Hence, we conclude that $d_{\F}(A',\hat{C}) \le 2(B_1+R_0)$. Thus the lemma holds with $D_1 = 2(B_1+ R_0)$.
\end{proof}

The proof of \Cref{prop: geos_fellow_travel} with take the rest of the section and require several lemmas. In fact, we first prove the theorem in the special case that $\rho$ is a standard geodesic (\Cref{prop: folding_paths_fellow_travel}) and complete the general proof in \Cref{sec:BF_projs_morse}. We note that only the special case is needed for the proof of our main result. 

First, we observe that quasigeodesics that make definite progress in the factor graph cannot become arbitrarily thin.

\begin{lemma} \label{lem: quasigeo_stays_thick}
Let $\gamma\colon \I \to \os$ be a $K$--quasigeodesic whose projection $\pi \circ \gamma\colon\I \to \fc$ is also a $K$--quasigeodesic. Then there is an $\epsilon >0$ depending only on $K$ so that $\gamma(i) \in \os_{\epsilon}$ for all $i\in \I$ with $i+K(K+11)\in \I$.

Furthermore, for any $i\in \I$ with $i+K(K+11)\notin \I$ (so that necessarily $\Ipl<\infty$), we have $\gamma(i)\in \os_{\epsilon'}$ for some $\epsilon'>0$ depending only on $K$ and the injectivity radius of $\gamma(\Ipl)$.
\end{lemma}
\begin{proof}
Since $\gamma$ is a $K$--quasigeodesic in $\fc$, we have  $d_{\fc}(\gamma(i),\gamma(j)) \ge \frac{1}{K}\abs{j-i} -K$. If $b \ge 0$ is chosen to equal  $K(K+11)$, then $\gamma(i)$ and $\gamma(i+b)$ have distance in the factor complex at least $11$. By \Cref{lem:easy}, this implies that there is no nontrivial conjugacy class having length less than $1$ in each of $\gamma(i), \gamma(i+b) \in \os$. If there exists $\alpha\in \free$ with $\ell(\alpha\vert(\gamma(i)) = \epsilon \le 1$, our choice of $b$ thus forces $\ell(\alpha\vert\gamma(i+b))\ge 1$. Hence we find that 
\[Kb+K \ge d_{\os}(\gamma(i),\gamma(i+b)) \ge \log\left(\frac{\ell(\alpha\vert \gamma(i+b))}{\ell(\alpha\vert\gamma(i))}\right) \ge \log\left(\frac{1}{\epsilon}\right).\]
This ensures $\epsilon \ge e^{-(Kb+K)}$, and so we conclude $\gamma(i)\in \os_{e^{-(Kb+K)}}$ for all $i\in \I$ with $i+b\in \I$.

Finally suppose $\Ipl<\infty$ and that $\gamma(\Ipl)\in \os_{\epsilon_0}$. If $i\in \I$ fails to satisfy $i+b\in \I$, then for any nontrivial $\alpha\in \free$ we similarly have
\[Kb+K \ge d_{\os}(\gamma(i),\gamma(\Ipl)) \ge \log\left(\frac{\ell(\alpha\vert\gamma(\Ipl))}{\ell(\alpha\vert\gamma(i))}\right) \ge \log\left(\frac{\epsilon_0}{\ell(\alpha\vert\gamma(i))}\right).\]
Thus $\ell(\alpha\vert \gamma(i))\geq \epsilon_0 e^{-(Kb+K)}$ for every nontrivial $\alpha\in \free$, which proves the claim.
\end{proof}

\begin{proposition}
\label{prop:fellow_travel_progressing_quasis} 
Let $\gamma\colon \I \to \os$ be a $K$--quasigeodesic whose projection $\pi \circ \gamma\colon \I \to \fc$ is also a $K$--quasigeodesic, and let $\rho\colon \J\to \os$ be a standard geodesic with the same endpoints as $\gamma$. Then there exists a constant $D_0 \ge 0$ depending only on $K$ so that
\begin{enumerate}
\item[(i)]\label{fellow_travel1} For all $i \in \I$ there is $t_i \in \J$ so that $d_{\os}(\gamma(i), \rho(t_i)) \le D_0$.
\item[(ii)]\label{fellow_travel2} For all $j \in \J$ there is $s_j \in \I$ so that $d_\os(\gamma(s_j),\rho(j))\le D_0$.
\end{enumerate}
\end{proposition}
\begin{proof}
We first prove (i). Let ${\sf B}$ be the universal constant from \Cref{prop:BF_folding_paths_contract}, and let $\tau' = R_0(\delta,\max\{K,K_f \})$ be the fellow-travelling constant (\Cref{prop:general_stability_for_quasis}) for $\max\{K,K_f\}$--quasigeodesics in a $\delta$--hyperbolic space (recall that $\fc$ is $\delta$--hyperbolic). Set $\tau = \tau' +D_1$, where $D_1$ is the constant appearing in \Cref{lem:folding_projection_is_closest_pt}. Define
\[ M = 2K^2{\sf B}+K,\quad C_0 = \frac{M-K}{K} = 2K{\sf B},\quad\text{and}\quad L_0 =2K({\sf B} + 2\tau + K).\]

Suppose that $[a,b]\subset\I$ is a subinterval such that $d_\os(\gamma(t),\foldim(\rho)) \ge M$ for all $t\in [a,b]$. Setting $n = \ceil{(b-a)/C_0}$, we then have the sequence of points $q_j = \gamma\left(a+\left(\frac{b-a}{n}\right)j\right)$ for $j = 0,\dotsc,n$. Notice that these points enjoy
\[d_\os(q_j,q_{j+1}) \leq K\left(\frac{b-a}{n}\right) + K  \le K(C_0) + K = M\]
for each $j = 0,\dotsc,n-1$. Since $d_\os(q_j,\foldim(\rho))\geq M$ by assumption, \Cref{prop:BF_folding_paths_contract} now implies that 
\[d_\fc\Big(\proj_{\rho}(q_j),\proj_{\rho}(q_{j+1})\Big) \leq {\sf B}\]
for all $j = 0,\dotsc,n-1$. 

Recall that by \Cref{lem:folding_projection_is_closest_pt}, $\pi(\proj_{\rho}(q_j))\in \pi(\foldim(\rho))$ is within distance $D_1$ from the closest point projection of $\pi(q_j)\in \fc$ to the path $\pi(\rho(\J))$. (Note that $\fproj(\rho(\J)) = \fproj(\foldim(\rho))$ since rescaling paths have constant projection in $\fc$ by construction.) Since $\pi\circ \gamma$ and $\pi\circ\rho$ are both (unparameterized) $\max\{K,K_f\}$--quasigeodesics and $\fc$ is $\delta$--hyperbolic, these paths have Hausdorff distance at most $\tau'$ in $\F$. It follows that for each $j = 0,\dotsc,n$ we also have $d_\fc\big(\pi(q_j),\pi(\proj_{\rho}(q_j))\big) \leq \tau'+D_1 = \tau$. By the triangle inequality, we now have
\begin{eqnarray*}
d_\fc(\gamma(a),\gamma(b)) 
&\leq& d_\fc(q_0,\proj_{\rho}(q_0)) + \sum_{j=0}^{n-1} d_\fc\Big(\proj_{\rho}(q_j),\proj_{\rho}(q_{j+1})\Big) + d_\fc(\proj_{\rho}(q_n),q_n)\\
&\leq&  n{\sf B} + 2\tau \\
&\leq&  \left(\frac{b-a}{C_0} + 1\right){\sf B} + 2\tau\\
&=&  \frac{b-a}{2K} + {\sf B} + 2\tau.
\end{eqnarray*}
On the other hand, by hypothesis we also have $d_\fc(\gamma(a),\gamma(b))\geq \frac{b-a}{K}- K$. Combining these, we find that
\[b-a \leq 2K({\sf B} + 2\tau + K) = L_0.\]
That is, $L_0$ is an upper bound for the length of any subinterval of $\I$ on which $\gamma$ stays at least distance $M$ from $\foldim(\rho)$. Said differently, for any $t\in \I$, there exists $0\leq t'\leq L_0$ so that 
$d_\os(\gamma(t+t'),\foldim(\rho)) < M$. (When $t+L_0\in \I$ this is clear. When $t+L_0\notin \I$, then we necessarily have $\Ipl<\infty$ and the assumption that $\gamma$ and $\rho$ have the same ends ensures $\gamma(\Ipl)\in \foldim(\rho)$.) In particular, we conclude that
\[d_\os(\gamma(t),\foldim(\rho))\leq d_\os\big(\gamma(t),\gamma(t+t')\big) + d_\os\big(\gamma(t+t'),\foldim(\rho)\big) \leq KL_0 + K + M.\]
This proves (i) with $D_0=KL_0 + K + M$.

We now prove (ii). Let $E_0$ denote the  maximum value of $D_0 = K L_0 + K +M$ and of the quasiconvexity constant ${\sf A}$ provided by \Cref{cor:right_balls_quasi-convex}. Note that $E_0\geq K$. For each point $i\in \I$, let 
\[U_i = \{y\in \foldim(\rho) \mid d_\os(\gamma(i),y)\leq 4E_0\}.\]
By the proof of (i), we know that there exists a point $y_i\in U_i$ with $d_\os(\gamma(i),y_i)\leq 2E_0$; in particular $U_i$ contains the length $2E_0$ subinterval of $\foldim(\rho)$ starting at $y_i$. Let $W_i\subset \foldim(\rho)$ denote the smallest connected interval containing $U_i$. It follows that each interval $W_i$ with $\rho(\Jpl)\notin W_i$ has length at least $2E_0$. By \Cref{cor:right_balls_quasi-convex} we additionally know that $d_\os(\gamma(i),w)\leq 4E_0 + {\sf A} \leq 5E_0$ for all $w\in W_i$.

Using that the projection $\pi \colon \X \to \F$ is coarsely $80$--Lipschitz (\Cref{lem: 80Lip}) we see that $\diam_\fc\fproj\left(\{\gamma(i)\}\cup W_i\right)\leq 80(10E_0)$. In particular, if $i,j\in \I$ satisfy $\abs{i-j}\geq 2\cdot80(10KE_0)$, then $d_\fc(\gamma(i),\gamma(j))\ge 2\cdot 80(10E_0)$ ensuring that $\fproj(W_i)$ and $\fproj(W_j)$ are disjoint. In particular, this implies $W_i\cap W_j = \emptyset$. On the other hand, if $i,j\in \I$ satisfy $i<j$ and $(j-i)\leq 1$, then
\[d_\os(\gamma(i),y_j)\leq d_\os(\gamma(i),\gamma(j)) + d_\os(\gamma(j),y_j) \leq K(1) + K + 2E_0 \leq 4E_0\]
showing that $y_j\in U_i$ by definition. Thus $W_i$ and $W_j$ intersect whenever $\abs{i-j}\leq 1$. This implies that the union
\[W = \cup_{i\in \I} W_i\]
is a connected subinterval of $\foldim(\rho)$. We claim that in fact $W = \foldim(\rho)$. 

To see this, first suppose $\Ipl<\infty$, in which case we also have $\Jpl<\infty$ and $\gamma(\Ipl) = \rho(\Jpl)\in \foldim(\rho)$ by assumption. In particular, $\rho(\Jpl)\in W_{\Ipl}\subset W$ by definition. If we instead have $\Ipl=\Jpl=\infty$, then the above shows that for any $t\in \J$ we can find infinitely many disjoint intervals $W_i\subset \rho([t,\infty))$ that each have length at least $2E_0$. Thus $W\cap \rho([t,\infty))$ is an infinite-length interval and so covers the positive end of $\foldim(\rho)$.

Now suppose $\Imin = \Jmin = -\infty$. In this case, we claim $\rho$ cannot have an initial rescaling segment (i.e., that $\scaleim(\rho) = \emptyset$ and consequently that $\rho(\J) = \foldim(\rho)$). Indeed, if $\scaleim(\rho)$ were nonempty then it must have infinite length in the negative direction. Since it is a rescaling path, this implies $\scaleim(\rho)$ contains arbitrarily thin points (\Cref{lem:bound_on_scaling_length}). However this contradicts the fact that $\gamma(\I)$ is contained in some thick part $\os_\epsilon$ (by \Cref{lem: quasigeo_stays_thick}) and that the initial rays of $\gamma$ and $\rho$ have finite Hausdorff distance. Therefore, $\foldim(\rho)$ has infinite length in the negative direction and the same argument as above shows that $W\cap \rho((-\infty,t])$ has infinite length for any $t\in \J$. Whence $W= \foldim(\rho)$ as claimed.

Finally suppose $\Imin\neq -\infty$. Let $t\in \J$ be such that $\rho(t) = y_{\Imin}\in U_{\Imin}\subset \foldim(\rho)$. Then $d_\os(\gamma(\Imin),\rho(t))\leq 4E_0$ by definition and, since $\rho$ is a geodesic, it follows that
\[d_\os(\gamma(\Imin),\rho(s)) = d_\os(\rho(\Jmin),\rho(s))\leq 4E_0\]
for all $s\in [\Jmin,t]$. In particular, $U_{\Imin}\subset W$ contains the left endpoint of $\foldim(\rho)$ which proves the desired equality $W = \foldim(\rho)$. Moreover, the above equation shows that any point $y\in \scaleim(\rho)$ satisfies $d_\os(\gamma(\Imin),y)\leq 4E_0$. Therefore we conclude that for every $s\in \J$ the point $\rho(s)\in \scaleim(\rho)\cup W$ satisfies $d_\os(\gamma(\I),\rho(s))\leq 5E_0$. Hence (ii) holds with $D_0 = 5E_0$.
\end{proof}

\begin{lemma}[Thinness prevents factor progress]
\label{lem:thin_progress_is_slow}
Suppose that $\gamma\colon[0,L]\to\os$ is a finite-length geodesic and that $\gamma(t)$ is $\epsilon$--thin for all $t\in [0,L]$. Then
\[d_\os(\gamma(0),\gamma(L)) \geq \log\left(\nicefrac{1}{\epsilon}\right)\frac{d_\fc(\gamma(0),\gamma(L))-20}{20}.\]
\end{lemma}
\begin{proof}
We may suppose $N = d_\fc(\gamma(0),\gamma(L))> 11$, for otherwise there is nothing to prove. Set $a_0 = 0$. Supposing by induction that $a_i\in[0,L)$ has been defined for some $i\geq 0$, we then set
\[a_{i+1} = \sup\{t\in [a_i,L] \mid d_\fc(\gamma(a_i),\gamma(t))\leq 15\}.\]
In this way, we obtain a sequence of times $0 = a_0 < \dotsb < a_n = L$. Notice that provided $a_{i+1}<L$, we necessarily have $d_\fc(\gamma(a_i),\gamma(a_{i+1}+\delta))\geq 16$ for all $\delta> 0$. Furthermore, for all sufficiently small $\delta$, the graphs $\gamma(a_{i+1})$ and $\gamma(a_{i+1}+\delta)$ necessarily have embedded loops representing the same conjugacy class, and so the projections $\pi(\gamma(a_{i+1}))$ and $\pi(\gamma(a_{i+1}+\delta))$ must overlap. Therefore the union of $\fproj(\gamma(a_{i}))$ and $\fproj(\gamma(a_{i+1}))$ has diameter at least $12$. By \Cref{lem:easy}, this implies that there is no nontrivial conjugacy class with length less than $1$ in both graphs $\gamma(a_i)$ and $\gamma(a_{i+1})$. Since by assumption $\ell(\beta\vert\gamma(a_i)) < \epsilon$ for some nontrivial $\beta\in \free$, it follows that $\ell(\beta\vert\gamma(a_{i+1}))\geq 1$ and thus that
\[d_\os(\gamma(a_i),\gamma(a_{i+1}))\geq \log\left(\frac{\ell(\beta\vert\gamma(a_{i+1}))}{\ell(\beta\vert\gamma(a_i))}\right)> \log\left(\nicefrac{1}{\epsilon}\right).\]
Therefore, since $\gamma$ is a geodesic, we find that
\begin{equation}\label{eqn:thin_prog_is_slow1}
d_\os(\gamma(a_0),\gamma(a_n)) = \sum_{i=0}^{n-1} d_\os(\gamma(a_i),\gamma(a_{i+1})) \geq (n-1)\log\left(\nicefrac{1}{\epsilon}\right).
\end{equation}

On the other hand, for each $i>0$ we can find arbitrarily small numbers $\delta > 0$ so that $d_\fc(\gamma(a_{i-1}),\gamma(a_i - \delta))\leq 15$. Since $\delta$ here can be taken arbitrarily small, it follows that $\gamma(a_{i} -\delta)$ and $\gamma(a_i)$ necessarily share an embedded loop. Consequently $\pi(\gamma(a_i))$ and $\pi(\gamma(a_i-\delta))$ overlap, and so we conclude
\[d_\fc(\gamma(a_{i-1}),\gamma(a_i)) \leq 20.\]
By the triangle inequality, it follows that
\begin{equation}\label{eqn:thin_prog_is_slow2}
d_\fc(\gamma(0),\gamma(L)) = d_\fc(\gamma(a_0),\gamma(a_n))\leq 20 n.
\end{equation}
Combining equations \eqref{eqn:thin_prog_is_slow1} and \eqref{eqn:thin_prog_is_slow2} gives the claimed result.
\end{proof}

\begin{lemma} \label{lem: folding_stays_thick}
Let $\gamma\colon\I\to \os$ be a $K$--quasigeodesic such that $\pi\circ\gamma\colon \I\to \fc$ is a $K$--quasigeodesic and $\gamma(\I)\subset \os_\epsilon$. Then there exists $\epsilon'>0$, depending only on $\epsilon$ and $K$, so that any standard geodesic $\rho\colon\J\to \os$ with the same endpoints as $\gamma$ is $\epsilon'$--thin, i.e.  $\rho(\J) \subset \os_{\epsilon'}$.
\end{lemma}
\begin{proof}
Let $E\geq 1$ be the maximum of $K$ and the constant $D_0$ provided by \Cref{prop:fellow_travel_progressing_quasis}, and choose $\epsilon_1\leq \epsilon$ sufficiently small so that $\log(\nicefrac{1}{\epsilon_1})\ge 40E^2$. Notice that $\epsilon_1$ depends only on $K$ and $\epsilon$. The facts that $\gamma(\I)\subset \os_\epsilon$ and that $\rho$ and $\gamma$ have finite Hausdorff distance (since they share the same endpoints) implies that there is some $\epsilon_0$ so that $\rho(\J)\subset \os_{\epsilon_0}$. Choosing $\epsilon_0 < \epsilon$, we then have $\rho(\J),\gamma(\I)\subset \os_{\epsilon_0}$. 

Let us write $G_t = \rho(t)$ for $t\in \J$. Suppose now that $(a',b')\subset \J$ is subinterval such that $G_t\notin \os_{\epsilon_1}$ for all $t\in (a',b')$ (i.e, $G_t$ has an immersed loop of length less than $\epsilon_1$). Since $\rho\vert_{[a',b']}$ is a geodesic, \Cref{lem:thin_progress_is_slow} implies that
\[d_\os(G_{a'}, G_{b'}) \geq \log\left(\nicefrac{1}{\epsilon_1}\right)\frac{d_\fc(G_{a'},G_{b'})-20}{20}.\]
By \Cref{prop:fellow_travel_progressing_quasis}, we can find points $a,b\in \I$ so that $d_\os(\gamma(a),G_{a'}) \leq E$ and \linebreak $d_\os(\gamma(b),G_{b'}))\leq E$. Together with the fact that $\fproj\colon \os\to \fc$ is coarsely $80$--Lipschitz, this implies 
\begin{eqnarray*}
d_\os(G_{a'}, G_{b'}) &\geq& \log\left(\nicefrac{1}{\epsilon_1}\right)\frac{d_\fc(\gamma(a),\gamma(b))-160E - 20}{20}\\
&\ge&  \log\left(\nicefrac{1}{\epsilon_1}\right)\frac{\frac{1}{E}\abs{b-a}-161E - 20}{20}.
\end{eqnarray*}
On the other hand, since $\gamma(a)$ and $G(a')$ are $\epsilon_0$--thick, we have $d_{\X}(G_{a'}),\gamma(a)) \le E\cdot \sym_{\epsilon_0}$, for $\sym_{\epsilon_0}$ as in \Cref{lem: symmetric_in_thick}. So by the triangle inequality,
\begin{eqnarray*}
d_{\X}(G_{a'}, G_{b'}) &\le&  d_{\X}(G_{a'}, \gamma(a))+ d_{\X}(\gamma(a),\gamma(b)) + d_{\X}(\gamma(b), G_{b'}) \\
&\le& E \cdot \sym_{\epsilon_0} + E\abs{b-a} + 2E.
\end{eqnarray*}

Combining these inequalities, and using $\log(\nicefrac{1}{\epsilon_1})\ge 40E^2$, we find that
\[\abs{b-a} \leq \sym_{\epsilon_0} + 2 + 322E^2+ 40E \]
By the triangle inequality it follows that
\begin{eqnarray*}
\abs{b'-a'} &\leq& d_\os(G_{a'},\gamma(a)) + d_\os(\gamma(a),\gamma(b)) + d_\os(\gamma(b),G(b'))\\
        &\leq& E\cdot \sym_{\epsilon_0} + E\abs{b-a} + E + E\\
        &\leq& 2E \cdot \sym_{\epsilon_0} + 4E +322E^3 +40E^2.
\end{eqnarray*}
In particular, this shows that $\J$ cannot contain an infinite length subinterval on which $\rho$ is $\epsilon_1$--thin. Thus $\mathbf{J'} \colonequals \{t\in \J \mid G_t\notin \os_{\epsilon_1}\}$ is a disjoint union of finite subintervals of $\J$. Each component of $\mathbf{J}$ thus has the form $(c',d')\subset \I'$ where $G_{c'},G_{d'}\in \os_{\epsilon_1}$ but $G_t\notin\os_{\epsilon_1}$ for all $t\in (c',d')$. (Note that if $\I_{\pm} \neq \pm\infty$, then $\gamma(\I_\pm)\in \os_{\epsilon_1}$ by choice of $\epsilon_1\leq \epsilon$.) Since $G_{c'}, G_{d'}\in \os_{\epsilon_1}$, a repetition of the above argument now implies
\[\abs{d'-c'} \leq L,\]
where $L\colonequals 2E \cdot \sym_{\epsilon_1} + 4E +322E^3 +40E^2$ depends only on $E$ and $\epsilon_1$ (and hence only on $K$ and $\epsilon$). Consequently, since $\rho$ is a geodesic, for any $t\in [c',d']$ and $\alpha\in \free$ we have
\[ \epsilon_1\leq  \ell(\alpha\vert G_{d'}) \le e^{L} \ell(\alpha\vert G_t),\]
which implies that $G_{t}\in \os_{\epsilon'}$ for $\epsilon'\colonequals \epsilon_1 e^{-L}$. Since this estimate holds for every point $t\in \mathbf{J'}$ and $\epsilon'$ depends only on $K$ and $\epsilon$, the result follows. \qedhere

\end{proof}

Before proving \Cref{prop: geos_fellow_travel} in its full generality, we focus on the case where the geodesic $\rho$ is a standard geodesic. 

\begin{proposition} \label{prop: folding_paths_fellow_travel} 
The conclusions of \Cref{prop: geos_fellow_travel} hold under the additional assumption that $\rho\colon \J\to \os$ is a standard geodesic.
\end{proposition}
\begin{proof}
Let $\gamma\colon \I \to \X$ be a $K$--quasigeodesic whose projection $\pi \circ \gamma\colon \I \to \F$ is a $K$--quasigeodesic, and let $\rho\colon \J\to \os$ be any standard geodesic with the same endpoints as $\gamma$. By \Cref{lem: quasigeo_stays_thick}, $\gamma$ is $\epsilon$--thick for some $\epsilon \ge 0$ depending only on $K$ (and on the injectivity radius of $\gamma(\I_+)$ when $\Ipl<\infty$). \Cref{lem: folding_stays_thick} therefore provides an $\epsilon' \ge 0$, depending only on $K$ and $\epsilon$, so that $\rho(t)\in X_{\epsilon'}$ for all $t\in \J$. Thus conclusion (i) holds.

Applying \Cref{prop:fellow_travel_progressing_quasis} in conjunction with the symmetrization estimate from \Cref{lem: symmetric_in_thick}, we see that for each $i \in \I$ there exists $t_i\in \J$  with $\dsym(\gamma(i),\rho(t_i)) \le \sym_{\epsilon'}D_0$. Similarly for every $j\in \J$ there is some $s_j\in \I$ so that $\dsym(\gamma(s_j),\rho(j))\le \sym_{\epsilon'}D_0$. Thus conclusion (ii) holds with $A = \sym_{\epsilon'}D_0$ since we have shown that
\[\dhaus(\gamma(\I),\rho(\J)) \le \sym_{\epsilon'}D_0.\]

It is now easy to see that $\fproj\circ\rho\colon \J\to \fc$ is a parameterized quasigeodesic: Consider any times $a,b\in \J$ with $a<b$. Since $\fproj$ is coarsely $80$--Lipschitz, we automatically have
\[d_\fc(\rho(a),\rho(b))\le 80\cdot d_\os(\rho(a),\rho(b)) + 80 =  80\abs{b-a} + 80.\]
On the other hand, by the above there exist times $s,t\in \I$ such that $\dsym(\gamma(s),\rho(a))$ and $\dsym(\gamma(t),\rho(b))$ are both bounded by $\sym_{\epsilon'}D_0$. By the triangle inequality, it follows that
\[d_\os(\gamma(s),\gamma(t)) \ge d_\os(\rho(a),\rho(b)) - 2\sym_{\epsilon'}D_0 = \abs{b-a} - 2\sym_{\epsilon'}D_0.\]
Since $\gamma$ is a directed $K$--quasigeodesic by assumption, this implies
\[(t-s) \ge \tfrac{1}{K}d_\os(\gamma(s),\gamma(t)) - K \ge \tfrac{1}{K}\abs{b-a} - \tfrac{2\sym_{\epsilon'}D_0}{K} - K.\]
Since $\fproj\circ\gamma\colon\J\to \fc$ is also a $K$--quasigeodesic, we may extend this to conclude
\[\abs{b-a} \le K(t-s) + 2\sym_{\epsilon'}D_0 + K^2 \le K\big(K d_\fc(\rho(a),\rho(b)) + K\big) + 2\sym_{\epsilon'}D_0 + K^2.\] 
Therefore, $\fproj\circ \rho$ is a $K'$--quasigeodesic for $K' = \max\big\{80,\; 2K^2 + 2\sym_{\epsilon'}D_0\big\}$. This proves conclusion (iii).
\end{proof}

\subsection{More on Bestvina--Feighn projections}
\label{sec:BF_projs_morse}

\Cref{prop: folding_paths_fellow_travel} above suffices to prove our main result on hyperbolic extensions of free groups (\Cref{th: intro_main_1}). However for completeness, and to strengthen the quasiconvexity results in \Cref{sec:qi_implies_qcx}, it is desirable to prove the more general result \Cref{prop: geos_fellow_travel} which applies to arbitrary geodesics. This subsection is devoted to that purpose.

Heuristically, \Cref{prop: geos_fellow_travel} follows easily from \Cref{prop: folding_paths_fellow_travel} and some ideas in Bestvina--Feighn \cite{BFhyp}. Specifically, as remarked in \cite[Corollary 7.3]{BFhyp}, Bestvina and Feighn's Proposition 7.2 (\Cref{prop:BF_folding_paths_contract} here) essentially says that folding paths that make definite progress in the factor complex are strongly contracting in Outer space, which generalizes Algom-Kfir's result \cite{AKaxis}. 
One should then apply this notion of strong contracting to conclude that such folding paths are stable (using standard arguments). However, to make this precise, we first require a more detailed discussion of the projection $\proj_\gamma\colon \os\to \gamma(\I) $.

Following \cite{BFhyp}, given a free factor $A\in \fc^0$ and a point $G\in \os$, we write $A\vert G$ for the core subgraph of the cover of $G$ corresponding to the conjugacy class of $A$ in $\free\cong\pi_1(G)$.   We say that $A|G$ is \define{the core of the $A$-cover}. Restricting the covering map thus gives a canonical immersion $A\vert G \to G$ that identifies $\pi_1(A\vert G)$ with $A\le \pi_1(G)$. The graph $A\vert G$ is equipped with a metric structure by pulling back the edge lengths from $G$. Similarly, whenever $G$ is given an illegal turn structure (e.g., if $G$ lies on a folding path), we may pull back this structure via $A\vert G\to G$, equipping $A\vert G$ with an illegal turn structure as well. When $A$ is a cyclic free factor generated by a primitive element $\alpha\in \free$, we note that  $A\vert G$ agrees with our already defined $\alpha\vert G$. 

Setting $I = (18\breve{m}(3r-3)+6)(2r - 1)$, where $r = \rank(\free)$ and $\breve{m}$ denotes the maximum number of illegal turns in any train track structure on any $G\in \os$, Bestvina and Feighn then define the following projections from $\F$ to folding paths in $\os$:

\begin{definition}
Let $\gamma\colon \I\to \os$ be a folding path, and let $A\in \fc^0$ be a proper free factor. The \define{left} and \define{right} projections of $A$ to $\gamma$ are respectively given by:
\begin{eqnarray*}
\mathrm{left}_\gamma(A) &\colonequals& \inf\{t \in \I : A\vert G_t \text{ has an immersed legal segment of length $3$}\}\in \I\\
\mathrm{right}_\gamma(A) &\colonequals& \sup\{t \in \I : A\vert G_t \text{ has an immersed illegal segment of length $I$}\}\in \I,
\end{eqnarray*}
where here an \define{illegal segment} means a segment that does not contain a legal segment of length $3$.
\end{definition}

Using this, the Bestvina--Feighn projection $\proj_\gamma$ is defined as follows:

\begin{definition}[Bestvina--Feighn projection]
\label{def:BestvinaFeighn_projection}
Let $\gamma\colon \I\to \os$ be a folding path. For $H\in \os$, the left and right projections of $H$ are defined to be
\[\mathrm{left}_\gamma(H) \colonequals \inf_{A\in \fproj(H)} \mathrm{left}_\gamma(A) \qquad\text{and}\qquad \mathrm{right}_\gamma(A) \colonequals \sup_{A\in \fproj(H)} \mathrm{right}_\gamma(A).\]
The \define{projection} of $H$ to $\gamma(\I)$ is then given by $\proj_\gamma(H) \colonequals \gamma(\mathrm{left}_\gamma(H))$.
\end{definition}

Note that every candidate conjugacy class $\alpha\in \mathcal{C}_H$ at $H\in \os$ is primitive and thus generates a cyclic free factor of $\free$; thus we may view $\alpha$ as a point in $\fc^0$. Since the immersion $\alpha\vert H\to H$ lands in a proper subgraph of $H$, we additionally have $\alpha \le A$ for some $A\in \fproj(H)$. Therefore, Bestvina and Feighn's Proposition 6.4 and Corollary 6.11 immediately give the following estimates regarding the above projections.

\begin{proposition}[Bestvina--Feighn]
\label{prop:projections_agree}
Let $\gamma\colon \I\to \os$ be a folding path and let $H\in \os$ be any point. Then for every candidate $\alpha\in \mathcal{C}_H$ of $H$, we have
\[\left[\mathrm{left}_\gamma(\alpha),\mathrm{right}_\gamma(\alpha)\right] \subset \left[\mathrm{left}_\gamma(H),\mathrm{right}_\gamma(H)\right] \subset \I.\]
Moreover, the set
\[\fproj\left(\gamma\Big(\big[\mathrm{left}_\gamma(H),\mathrm{right}_\gamma(H)\big]\Big)\right)\subset \fc\]
has uniformly bounded diameter depending only on $\rank(\free)$.
\end{proposition}

As a consequence, we may deduce that $\proj_\gamma(H)$ coarsely agrees with the closest point projection of $H$ to $\gamma(\I)$ in the case that $\gamma$ makes definite progress in $\fc$.

\begin{lemma}
\label{lem:coarsely closest point}
Let $\gamma\colon \I\to \os$ be a folding path whose projection $\fproj\circ\gamma\colon \I\to \fc$ is a $K$--quasigeodesic. Then there exists $D\ge 0$, depending only on $K$ and $\rank(\free)$ (and the injectivity radius of $\gamma(\Ipl)$ when $\Ipl<\infty$) satisfying the following: If $H\in \os$ and $t_0\in \I$ are such that
\[d_\os(H,\gamma(t_0)) = \inf\big\{ d_\os(H,\gamma(t)) \mid t\in \I\big\},\]
then $\dsym(\gamma(t_0),\proj_\gamma(H)) \le D$.
\end{lemma}
\begin{proof}
We write $G_t = \gamma(t)$ for $t\in \I$. Let us define
\[L = \inf \big\{\mathrm{left}_\gamma(\alpha) \mid \alpha\in \mathcal{C}_H\big\} \qquad\text{and}\qquad R = \sup \big\{\mathrm{right}_\gamma(\alpha) \mid \alpha\in \mathcal{C}_H\big\}.\]
Note that each candidate $\alpha\in \mathcal{C}_H$ is a simple class and that, by definition of $\mathrm{left}_\gamma(\alpha)$, the loop $\alpha\vert G_s$ cannot contain a legal segment of length $3$ for any $s<L$. Therefore, Lemma 5.8 of \cite{BFhyp} and the fact that $\fproj\circ\gamma$ is a $K$--quasigeodesic together imply that that there exists $T\ge 0$ depending only on $K$ and $\rank(\free)$ such that for all $t \ge T$ we have
\[\ell(\alpha \vert G_{L-t}) > 2 \ell(\alpha\vert G_L).\]
Since this estimate holds for each candidate, \Cref{pro: distance} implies that $2d_\os(H,G_L) < d_\os(H,G_{L-t})$ for all $t \ge T$.
Similarly, for all $s > R$ the loop $\alpha\vert G_s$ contains immersed legal segments contributing to a definite fraction of $\ell(\alpha\vert G_s)$. Therefore, by Corollary 4.8 of \cite{BFhyp}, the length $\ell(\alpha\vert G_s)$ grows exponentially beyond $R$ and so after increasing $T$ if necessary we have
\[\ell(\alpha \vert G_{R+t}) > 2 \ell(\alpha\vert G_R)\]
and consequently $2d_\os(H,G_R) < d_\os(H,G_{R+t})$ for all $t\ge T$. Given any time $t_0\in \I$ satisfying
\[d_\os(H,G_{t_0}) = \inf\big\{ d_\os(H,\gamma(t)) \mid t\in \I\big\},\]
it follows that $t_0$ necessarily lies in $[L-T,R+T]$.

By \Cref{prop:projections_agree}, we know that $\fproj(\gamma([L,R]))$ has bounded diameter and bounded $\fc$--distance from $\fproj(\proj_\gamma(H))$. Therefore, since $\fproj\circ\gamma$ is a $K$--quasigeodesic, there exists $D'$, depending only on $K$ and $\rank(\free)$, so that $\abs{s_0-t_0}\le D'$, where $s_0\in \I$ is the time for which $G_{s_0}=\proj_\gamma(H)$. By \Cref{lem: quasigeo_stays_thick}, we additionally know $\gamma(\I)\subset \os_\epsilon$ for some $\epsilon>0$ depending on $K$ (and the injectivity radius of $\gamma(\Ipl)$ when $\Ipl<\infty$). Therefore, since $\gamma$ is a directed geodesic, we may conclude $\dsym(G_{t_0},\proj_\gamma(H)) \le \sym_{\epsilon} D'$, as desired.
\end{proof}

\Cref{lem:coarsely closest point} shows that whenever $\gamma\colon \I\to \os$ is a standard geodesic for which $\fproj\circ\gamma$ is a $K$--quasigeodesic, then the closest point projection $\os\to \gamma(\I)$ coarsely agrees with $\proj_\gamma\colon \os\to \gamma(\I)$. Thus, since $\gamma$ makes definite progress in $\fc$, \Cref{prop:BF_folding_paths_contract} implies that $\gamma$ is strongly contracting. That is, there exists $D$, depending only on $\rank(\free)$ and $K$ (and the injectivity radius of $\gamma(\Ipl)$ if $\Ipl<\infty$), such that if $d_\os(H,H') \le d_\os(H,\gamma(\I))$, then any closest point projections of $H$ and $H'$ to $\gamma(\I)$ are at most $\dsym$--distance $D$ apart. We are therefore in the situation of the standard Morse lemma (see, e.g., Section 5.4 of \cite{AKaxis}), which gives the following stability result.

\begin{lemma}[Morse lemma for $\fc$--progressing folding paths]
\label{lem:morse lemma}
Suppose that $\gamma\colon \I\to \os$ is a standard geodesic for which $\fproj\circ\gamma\colon \I\to \fc$ is a $K$--quasigeodesic. Then for any $K'\ge 1$ there exists $B$ depending only on $\rank(\free)$, $K$, and $K'$ (and the injectivity radius of $\gamma(\Ipl)$ when $\Ipl<\infty)$  such that $\dhaus(\gamma(\I),\rho(\J))\le B$ for every $K'$--quasigeodesic $\rho\colon \J\to \os$ with the same endpoints as $\gamma$.
\end{lemma}

Using this, we may finally give the proof of \Cref{prop: geos_fellow_travel}:

\begin{proof}[Proof of \Cref{prop: geos_fellow_travel}]
Let $\gamma\colon \I\to \os$ be a $K$--quasigeodesic such that $\fproj\circ\gamma$ is also a $K$--quasigeodesic, and let $\epsilon,A> 0$ and $K'\ge 1$ be the corresponding constants provided by \Cref{prop: folding_paths_fellow_travel}. Choose a standard geodesic $\rho'\colon \J'\to \os$ with the same endpoints as $\gamma$. Then by \Cref{prop: folding_paths_fellow_travel} we know that $\rho'(\J')\subset \os_\epsilon$ and that $\fproj\circ\rho'$ is a $K'$--quasigeodesic. Now consider an arbitrary geodesic $\rho\colon \J\to \os$ with the same endpoints as $\gamma$, and thus also $\rho'$. Applying \Cref{lem:morse lemma} to $\rho$ and the folding path $\rho'$, we find that
\[\dhaus(\rho(\J),\rho'(\J'))\le B\]
for some $B$ depending only on $\epsilon$ and $K'$. Consequently $\rho(\J)\subset\os_{\epsilon'}$ where $\epsilon' = e^{-B}\epsilon$. Since $\rho'(\J')$ and $\gamma(\I)$ have Hausdorff distance at most $A$ by \Cref{prop: folding_paths_fellow_travel}, it also follows that $\dhaus(\rho(\J),\gamma(\I))\le B + A$. Finally, as in the proof of \Cref{prop: folding_paths_fellow_travel} above, these two facts easily show that $\fproj\circ\rho$ is a $K''$--quasigeodesic for some $K''$ depending only on $\epsilon'$ and $A+B$.
\end{proof}

\section{Quasi-isometric into $\F$ implies quasiconvex in $\X$}
\label{sec:qi_implies_qcx}

Consider a finitely generated subgroup $\Gamma\le \Out(\free)$. For any finite generating set $S\subset \Gamma$, we then consider the word metric $d_\Gamma = d_{\Gamma,S}$ on $\Gamma$ defined by $d_\Gamma(g,h) = \abs{g\inv h}_S$, where $\abs{\cdot}_S$ denotes word length with respect to $S$. This is just the restriction of the path metric on the Cayley graph $\cay{S}{\Gamma}$ to $\Gamma = (\cay{S}{\Gamma})^{0}$. In this section we explain various ways in which the geometry of $\Gamma$ relates to that of $\os$ or $\fc$.

For any free factor $A\in \fc^0$, we may consider the orbit map $(\Gamma,d_\Gamma)\to (\fc,d_\fc)$ given by $g\mapsto g\cdot A$. We say that this map is a qi-embedding if it is a $K$--quasi-isometric embedding for some $K\ge 1$. We remark that if some orbit map into $\fc$ is a quasi-isometric embedding, then so is any orbit map into $\fc$.

\begin{definition}\label{def:qi_into_fc}
We say $\Gamma\le \Out(\free)$ \define{qi-embedds into $\fc$} if $\Gamma$ is finitely generated and any orbit map into $\fc$ is a qi-embedding.
\end{definition}

Given a point $H\in \os$, we say that the orbit $\Gamma\cdot H$ is \define{quasiconvex} if it is $A$--quasi-convex for some $A\ge 0$, meaning that every (directed) geodesic between points of $\Gamma\cdot H$ lies in the (symmetric) $A$--neighborhood $\nbhd{A}{\Gamma\cdot H}$ (see \Cref{sec:metric_props_of_os}). We record the following straightforward consequence of quasiconvexity.

\begin{lemma}\label{lem:os_qconv_implies_os_qi}
Let $\Gamma\le\Out(\free)$ be finitely generated with corresponding word metric $d_\Gamma$, and suppose $H\in \os$ is such that $\Gamma\cdot H\subset \os$ is quasiconvex. Then the orbit map $g\mapsto g\cdot H$ defines a quasi-isometric embedding $(\Gamma,d_\Gamma)\to (\os,d_\os)$.
\end{lemma}
\begin{proof}
Let $S\subset \Gamma$ be the generating set inducing the word metric $d_\Gamma$. By assumption, there exists $A\ge 0$ so that $\Gamma\cdot H$ is $A$--quasiconvex. Choose $\epsilon = \epsilon(H,A)> 0$ so that $\nbhd{A}{\Gamma\cdot H}\subset \os_\epsilon$. Since $\Out(\free)$ acts properly discontinuously on $\os$, the set 
\[D = \{g\in \Gamma \mid \dsym(H,g\cdot H) \le 2A+\sym_\epsilon\}\]
is finite, and we may set $K = \max_{g\in D} d_\Gamma(1,g)$.

Letting $\gamma\colon[0,L]\to \os$ be a (directed) geodesic from $g\cdot H$ to $g'\cdot H$, our hypothesis implies $\gamma\subset \nbhd{A}{\Gamma\cdot H}$ and consequently that $\gamma(t)\in \os_\epsilon$ for all $t\in [0,L]$. Setting $N = \floor{L}$, we may find $h_0,\dotsc, h_{N+1}\in \Gamma$ so that $h_0 = g$, $h_{N+1} = g'$ and $\dsym(\gamma(i),h_i\cdot H)< A$ for all $i=0,\dotsc,N$. In particular, we see that for each $i=0,\dotsc, N$ the element $h_i\inv h_{i+1}$ translates $H$ by at most $\dsym$--distance $2A+\sym_{\epsilon}$ and therefore has $d_\Gamma(1,h_i\inv h_{i+1})\le K$. Thus
\begin{align*}
d_\Gamma(g,g') &\le d_\Gamma(h_0,h_1) + \dotsb + d_\Gamma(h_{N},h_{N+1})= \sum_{i=0}^N d_\Gamma(1,h_i\inv h_{i+1}) \le K(N+1)\\
&\le K(L+1) = K d_\os(g\cdot H, g'\cdot H) + K.
\end{align*}
On the other hand, if $K' = \max\{d_\os(H,s\cdot H) \vert s\in S\}$, then $d_\os(g\cdot H,g' \cdot H) \le K' d_\Gamma(g,g')$. Therefore $g\mapsto g\cdot H$ is a $\max\{K',K\}$--quasi-isometric embedding.
\end{proof}

\begin{definition}\label{def:quasi-convex}
A subgroup $\Gamma\le \Out(\free)$ is said to be \define{quasiconvex in $\os$} if the orbit $\Gamma\cdot H$ is quasiconvex for every $H\in \os$.
\end{definition}

We remark that knowing a single orbit $\Gamma\cdot H$ is quasiconvex in $\os$ does not necessarily seem to imply that $\Gamma$ is quasiconvex: it is conceivable that some other orbit $\Gamma\cdot H'$ could fail to be quasiconvex.

We now employ the results of \Cref{sec: progression} to show that every subgroup that qi-embedds into the factor complex is quasiconvex in Outer space:

\begin{theorem}
\label{thm:qi_into_factor_implies_quasiconvex}
Let $\Gamma\le \Out(\free)$ be finitely generated. If $\Gamma$ qi-embedds into $\fc$, then $\Gamma$ is quasiconvex in $\os$.
\end{theorem}

\begin{proof} 
Let $H \in \X$ be arbitrary and let $A \in \pi(H) \subset \F$. Since $\pi \colon \X \to \F$ is coarsely Lipschitz and $g\mapsto g\cdot A$ gives a quasi-isometric embedding $\Gamma\to \fc$, the orbit map $\mathcal{O} \colon \Gamma \to \X$ defined by $\mathcal{O}(g) = g\cdot H$ is also a quasi-isometric embedding. Let $g_1,g_2\in \Gamma$ be given. For any (discrete) geodesic path $\rho\colon \{1,\dotsc,N\}\to \Gamma$ from $g_1$ to $g_2$, the image $\mathcal{O} \circ \rho$ is thus a quasigeodesic path in $\X$ joining $g_1\cdot H$ and $g_2\cdot H$ such that $\fproj\circ \mathcal{O} \circ \rho$ is also a quasigeodesic in $\F$. \Cref{prop: geos_fellow_travel} then implies that any geodesic $\gamma\colon \I\to \os$ from $g_1\cdot H$ to $g_2\cdot H$ stays uniformly close to  the image of $\mathcal{O} \circ p$, which is contained in $\Gamma \cdot H$. Hence, $\Gamma$ is quasiconvex in $\X$.
\end{proof}

\section{Quasiconvex orbit implies conjugacy flaring}
\label{sec: qcx_implies_flar}

Consider a subgroup $\Gamma\le \Out(\free)$ with finite generating set $S\subset \Gamma$ and corresponding wordlength $\abs{\cdot}_S$. Fix also a basis $X$ of $\free$. We say that $\Gamma$ has \define{$(\lambda ,M)$--conjugacy flaring} for the given $\lambda > 1$ and positive integer $M\in \N$ if the following condition is satisfied:
\begin{itemize}
\item[] For all $\alpha \in \free$ and $g_1,g_2 \in \Gamma$ with $\abs{g_i}_S \ge M$ and $\abs{g_1g_2}_S = \abs{g_1}_S +\abs{g_2}_S$, we have
\[\lambda \norm{\alpha}_X \le \max \left\{\norm{g_1(\alpha)}_X, \norm{g_2\inv(\alpha)}_X\right\},\]
where $\norm{\cdot}_X$ denotes conjugacy length (i.e., the shortest word length with respect to $X$ of any element in the given conjugacy class).
\end{itemize}

In this section we show that any purely hyperbolic subgroup $\Gamma\le \Out(\free)$ that qi-embedds into $\fc$ has conjugacy flaring. In fact, our argument only relies on the following weaker hypothesis. Before making the definition, we first recall that a (finite) geodesic in $\Gamma$ may be encoded by a sequence of group elements $(g_0,\dotsc,g_N)$ such that $d_\Gamma(g_i,g_j) = \abs{i-j}$ for all $i,j=0,\dotsc,N$. For $R\in \os$, the \define{image of this geodesic} in the orbit $\Gamma\cdot R$ is simply the set of points $g_0\cdot R,\dotsc, g_N\cdot R$. 

\begin{definition}[QCX]\label{def:qcx}
Consider a subgroup $\Gamma\le \Out(\free)$ and point $R\in \os$. We say that the orbit $\Gamma\cdot R$ is \define{$A$--QCX} if for any geodesic $(g_0,\dotsc,g_N)$ in $\Gamma$ there exists a folding path $\rho\colon \J\to \os$ that has Hausdorff distance at most $A$ from the image of $(g_0,\dotsc, g_N)$, that is 
\[\dhaus\big(\rho(\J), \{g_0\cdot R,\dotsc, g_N\cdot R\}\big)\le A,\]
such that $\dsym(\rho(\Jmin),g_0\cdot R)\le A$ and $\dsym(\rho(\Jpl),g_N\cdot R)\le A$.
\end{definition}

We summarize this property by saying the image of the geodesic $(g_0,\dotsc,g_N)$ in $\Gamma\cdot R\subset \os$ has Hausdorff distance at most $A$ from a folding path in $\os$ with the correct orientation. Note that for an arbitrary subgroup $\Gamma$, there is no direct correspondence between quasi-convexity and this QCX condition. However, we have the following relationship when $\Gamma$ is hyperbolic.

\begin{lemma}\label{lem:hyperbolic_and_qconvex_implies_qcx}
Suppose that $\Gamma \le \Out(\free)$ is finitely generated, $\delta$--hyperbolic, and that $\Gamma \cdot R\subset \os$ is $A$--quasiconvex. Then $\Gamma\cdot R$ is $A'$--QCX for some $A'$.
\end{lemma}
\begin{proof}
  Let $d_\Gamma$ be a word metric on $\Gamma$ so that $(\Gamma,d_\Gamma)$ is $\delta$--hyperbolic. By \Cref{lem:os_qconv_implies_os_qi} the orbit map $g\mapsto g\cdot R$ defines a $K$--quasi-isometric embedding $(\Gamma,d_\Gamma)\to (\os,d_\os)$ for some $K$. Let $(g_0,\dotsc,g_N)$ be any geodesic in $\Gamma$ and let $\gamma_0\colon \I_0\to \os$ be a standard geodesic from $g_0\cdot R$ to $g_N\cdot R$. Then by quasiconvexity we have that $\gamma_0(\I_0)\subset\nbhd{A}{\Gamma\cdot R}$. Note that $\nbhd{A}{\Gamma\cdot R}\subset \os_\epsilon$ for some $\epsilon >0$ (since $R$ has positive injectivity radius). The scaling image $\scaleim(\gamma_0)$ of $\gamma_0$ therefore lives in $\os_\epsilon$ and thus has length at most $\log(2/\epsilon)$ by \Cref{lem:bound_on_scaling_length}. Setting $A_0 = A+\sym_\epsilon(\log(2/\epsilon) + 1)$, it follows that if $\gamma\colon [0,L]\to \os$ is the folding portion of $\gamma_0$ and $m = \floor{L}$, then we may find group elements $h_0,\dotsc, h_m$ with $h_0 = g_0$ and $h_m = g_N$ such that $\dsym(\gamma(i),h_i\cdot R) \le A_0$ for all $i=0,\dotsc, m\in[0,L]$. It follows that for $i<j$
\[ j-i - 2A_0 \le d_\os (h_i\cdot R, h_j\cdot R) \le j-i + 2A_0. \]
Therefore the map $i\mapsto h_i \cdot R$ is a discrete $2A_0$--quasigeodesic in $(\os,d_\os)$; consequently, the sequence $g_0 = h_0,\dotsc,h_m = g_N$ is a $K'$--quasigeodesic in $\Gamma$ for some $K' = K'(K,A_0)$. Since $\Gamma$ is $\delta$--hyperbolic, \Cref{prop:general_stability_for_quasis} implies that for each $j\in \{0,\dotsc, m\}$ there exists $i\in \{0,\dotsc, N\}$ with $d_\Gamma(h_j,g_i)\le R_0 = R_0(K',\delta)$. Noting that $\dsym(h_j\cdot R,g_i\cdot R) \le \sym_\epsilon(K d_\Gamma(h_j,g_i) + K)$, it follows that
\[\{h_0\cdot R,\dotsc, h_m\cdot R\}\subset \nbhd{\sym_\epsilon(KR_0+K)}{\{g_0\cdot R,\dotsc, g_N\cdot R\}}.\]
As we also have $\gamma([0,L])\subset \nbhd{\sym_\epsilon + A_0}{\{h_0\cdot R,\dotsc, h_m\cdot R\}}$ by the selection of $h_0,\dotsc, h_m$, the claim follows with $A' = \sym_\epsilon(KR_0+ K) + \sym_\epsilon+A_0$.
\end{proof}

\begin{corollary}\label{cor:fc-qi_implies_qcx}
If $\Gamma\le \Out(\free)$ qi-embedds into $\fc$, then for every $R\in\os$ there exists $A\ge 0$ so that the orbit $\Gamma\cdot R$ is $A$--QCX.
\end{corollary}
\begin{proof}
By \Cref{thm:qi_into_factor_implies_quasiconvex} we know that every orbit $\Gamma\cdot R$ is quasiconvex in $\os$. Since $\fc$ is hyperbolic, the hypothesis that $\Gamma$ qi-embedds into $\fc$ also implies that $\Gamma$ is finitely generated and  $\delta$--hyperbolic for some $\delta\ge 0$. \Cref{lem:hyperbolic_and_qconvex_implies_qcx} thus implies the claim.
\end{proof}

We also have the following simple consequence of being $A$--QCX:

\begin{lemma}\label{lem:qcx-implies_qi_in_os}
Suppose $\Gamma\le \Out(\free)$ is finitely generated and that the orbit $\Gamma\cdot R\subset \os$ is $A$--QCX. Then $g\mapsto g\cdot R$ gives a quasi-isometric embedding $(\Gamma,d_\Gamma)\to (\os,d_\os)$.
\end{lemma}
\begin{proof}
Let $g_1,\dotsc,g_N$ be a geodesic in $\Gamma$ from $g = g_1$ to $g' = g_N$. By using a folding path $\gamma\colon \I\to \os$ with Hausdorff distance at most $A$ from the image of $(g_1,\dotsc,g_N)$, an argument exactly as in \Cref{lem:os_qconv_implies_os_qi} shows that $d_\Gamma(g,g')$ and $d_\os(g\cdot R,g'\cdot R)$ agree up bounded additive and multiplicative error depending only on $R$ and $A$.
\end{proof}

Having established this terminology, we now turn to the main result of this section:

\begin{theorem}\label{qcx_implies_flaring}
Suppose that $\Gamma\le \Out(\free)$ is finitely generated,  purely hyperbolic, and that for some $R\in \os$ the orbit $\Gamma\cdot R$ is $A$--QCX. Then $\Gamma$ has $(2, M)$--conjugacy flaring for some $M\in \N$ depending only on $A$ and $R$.
\end{theorem}

The proof of \Cref{qcx_implies_flaring} will take several steps. We first show in \Cref{flaring_along_folding} that, provided $\Gamma$ is purely hyperbolic, a corresponding flaring property holds for the length of any conjugacy class along any folding path that remains within the symmetric $A$--neighborhood of the orbit $\Gamma \cdot R \subset \X$. When the orbit $\Gamma\cdot R$ is $A$--QCX we use this flaring on folding paths to deduce a similar flaring in the orbit $\Gamma\cdot R$. Measuring this flaring from $R$, where $\ell(\cdot \vert R)$ coarsely agrees with the conjugacy length $\norm{\cdot}_X$, then yields \Cref{qcx_implies_flaring}. We first require the following lemma, which is central to this section. It implies that there is a uniform bound on how long a conjugacy class can stay short along our folding paths.

\begin{lemma}\label{not_short_for_long}
Fix $\Gamma\le \Out(\free)$ and $R\in \os$. For any $L_0 \ge 0$ and $A_0 \ge 0$, there is a $D_0 \ge 0$ satisfying the following:  If  $\alpha \in \free$ is nontrivial and $\gamma\colon \I\to \os$ is a folding path with $G_t = \gamma(t)\in \nbhd{A_0}{\Gamma \cdot R}$ for all $t\in \I$, then either
\[\mathrm{diam}\{t\in \I : \ell(\alpha\vert G_t) \le L_0\} \le D_0\]
or there is an \emph{infinite order} element $\phi \in \Gamma$ with $\phi([\alpha]) =[\alpha]$. 
\end{lemma}
\begin{proof}
Let $\Gamma_{\alpha}$ be the subgroup of elements of $\Gamma$ that fix the conjugacy class of $\alpha$. If $\Gamma_{\alpha}$ is a torsion group, then $\abs{\Gamma_{\alpha}} \le e_r$ by \Cref{lem: phyp_implies_finiteprobs}.
 
Let $a$ and $b$ be the infimum and supremum of the set $\{t\in \I : \ell(\alpha\vert G_t) \le L_0\}$. Then, by \Cref{prop:lenght-quasi-convexity}, for all $t \in [a,b]$ we have $\ell(\alpha\vert G_t) \le ML_0$, where $M = 6\rank(\free)$. It follows that if $d_0 \ge 3A_0$, then for all $t,t+d_0 \in [a,b]$ the points $G_t$ and $G_{t+d_0}$ cannot both be $A_0$--close (in symmetric distance) to the same orbit point of $\Gamma\cdot R$ (since $\dsym(G_t,G_{t+d_0}) \ge d_\os(G_t,G_{t+d_0}) = d_0 > 2A_0$). 

Set $N = \floor{(b-a)/d_0}$ and for each $0 \le n \le N$ select $\phi_n \in \Gamma$ such that 
\[\dsym(\phi_n\cdot R, G_{a+d_0n}) \le A_0.\]
 By our choice of $d_0$, $\phi_i = \phi_j$ for $0 \le i,j \le N$ if and only if $i =j$. By assumption,  $\alpha \in \free$ has length at most $ML_0$ in $G_{a+d_0n}$; thus we have $\ell(\phi_n\inv(\alpha)\vert R) \le e^{A_0}ML_0$ for all $0\le n \le N$. Let $C$ denote the number of immersed loops in $R$ of length at most $e^{A_0}ML_0$; we note that $C$ depends only on $R$, $A_0$ and $L_0$. It follows that if $N> C(e_r+1)$ then we may find distinct $0 \le k_0 < \dotsb < k_{e_r} \le N$ such that
\[ \phi_{k_{0}}\inv (\alpha) = \phi_{k_{1}}\inv(\alpha) = \dotsb = \phi_{k_{e_r}}\inv(\alpha).\]
Since the $\phi_{k_i}$ are all distinct, this implies that $\Gamma_{\alpha}$ contains at least $e_r+1$ elements and, hence, an infinite order element. Otherwise $N\le C(e_r+1)$ and thus we conclude
\[ b-a \le d_0(N+1) \le d_0(C(e_r+1) +1).\]
Setting $D_0 =  d_0(C(e_r+1) +1)$ completes the proof.
\end{proof}

We next examine how the length of a loops varies over a folding path $G_t$ that is near the orbit of $\Gamma$. Our arguments are inspired by Section $5$ of \cite{BFhyp}, however, the use of \Cref{not_short_for_long} greatly simplifies our analysis. 

For a folding path $G_t$ and a conjugacy class $\alpha$, recall that $\alpha|G_t$ is the core of the $\alpha$-cover of $G_t$. We think of $\alpha|G_t$ as having edge lengths and illegal turn structure induced from $G_t$. As such, $\alpha|G_t$ is composed of legal segments separated by illegal turns. We say that a collection of consecutive illegal turns in $\alpha|G_t$ \define{survive} to $\alpha|G_{t'}$ for $t \le t'$ if no illegal turn in the collection becomes legal in the process of folding from $G_t$ to $G_{t'}$ nor do two illegal turns of the collection collide. In other words, a collection of consecutive illegal turns of $\alpha|G_t$ survive to $\alpha|G_{t'}$ if and only if there is a collection of consecutive illegal turns of $\alpha|G_{t'}$ and a bijection between the illegal turns in both collections induced by the process of unfolding an illegal turn of $\alpha|G_{t'}$ to an illegal turn of $\alpha|G_t$ (see \Cref{lem:unfolding} and the surrounding discussion). Set $\breve{m}$ equal to the maximum number of illegal turns in any train track structure on any $G \in \X$. Note that $\breve{m} \ge 2\rank(\free)-2$.

\begin{lemma}[Illegal turns don't survive] \label{lem: dont_survive}
Suppose that $\Gamma\le \Out(\free)$ is purely hyperbolic and that $R\in \os$. For each $l \ge 0$ and $A_0\ge 0$ there exists $D_l \ge 0$ satisfying the following property. If $G_t$ is a folding path with $G_t\in \nbhd{A_0}{\Gamma\cdot R}$ for all $t\in [a,b]$ and $\alpha$ is a conjugacy class such that $\alpha|G_{a}$ has a segment containing $\breve{m} +1$ consecutive illegal turns that survive to $\alpha|G_{b}$ and the length of each legal segment between these illegal turns in $\alpha|G_{b}$ is no greater than $l$, then $b-a \le D_l$.
\end{lemma}

\begin{proof}
Let $s_t$ be the segment spanning the consecutive surviving illegal turns in $\alpha|G_t$ for $a\le t \le b$. Since the number of illegal turns in $s_b$ is greater than the total number of illegal turns in $G_b$, there are a pair of illegal turns of $s_b$ that project to the same illegal turn of $G_b$ under the immersion $\alpha|G_b \to G_b$. Let $s'_b$ be the subsegment between two such turns and let $\sigma_b$ denote the loop obtained by projecting $s'_b$ to $G_b$ and identifying its endpoints.

By the unfolding principle of \cite{BFhyp} (\Cref{lem:unfolding}), there is a subsegment $s'_t$ of $s_t$ that maps to the segment $s'_b$ after folding and tightening and such that the illegal turn endpoints of $s'_t$ map to the same illegal turn in $G_t$ (just as in $G_b$). Hence, we may form the loop $\sigma_t$ by identifying these endpoints in $G_t$. We note for each $a\le t \le b$, $\sigma_t$ is immersed except possibly at the illegal turn corresponding to the endpoints of $s'_t$ and that the conjugacy class of $\sigma_t$ maps to the conjugacy class of $\sigma_b$ under the folding map $G_t \to G_b$, again by the unfolding principle. Let $\sigma$ denote this conjugacy class in $\free$.

By construction, the length of $\sigma_b$ is bounded by $l\cdot(\breve{m}+1)$ and the number of illegal turns of $\sigma_a$ is no more than $\breve{m}+1$, since these illegal turns all survive in $G_b$ by assumption. By \Cref{lem: length_v_legality}, $\ell(\sigma_t) \le 2l\cdot(\breve{m}+1)$ for all $a\le t \le b$. Then, by \Cref{not_short_for_long} either $\phi(\sigma) = \sigma$ for some infinite order $\phi \in \Gamma$ or we have $b-a\le D_l$ for some $D_l$ depending only on $A_0$, $l$ and $R$. Since $\Gamma$ is purely hyperbolic, the claim follows.
\end{proof}

Recall the notation from \Cref{sec:folding_path_quasiconvexity}: If $G_t$ is a folding path and $\alpha$ is a conjugacy class, then $k_t = k(\alpha|G_t)$ denotes the number of illegal turns of $\alpha|G_t$ and $m_t$ denotes the illegality of $G_t$. The following lemma is similar to Lemma $5.4$ of \cite{Brink}. Again, we use that our folding path in near the orbit of $\Gamma$ as a a replacement for having a single train track map, as was the case in \cite{Brink}. Let $r = \rank(\free)$.

\begin{lemma} \label{lem: turns_to_length}
Let $G_t$ be a folding path with $G_t \in \nbhd{A_0}{\Gamma \cdot R}$ for $t\in [a,b]$ and let $p_b$ be an immersed path in $G_{b}$ whose endpoints are illegal turns such that $k(p_b) \ge 2(2r-2)$ and $p_b$ contains no legal segment of length $L\ge 3$.  Let $p_t$ be the corresponding path in $G_t$ whose endpoints are illegal turns which is obtained from $p_b$ by unfolding.
Then
\[\frac{\epsilon_0 \cdot k(p_t)}{2(2r-2)} \le \ell(p_t) \le L\cdot k(p_t),\]
where $\epsilon_0$ is the minimal injectivity radius of any graph in $\nbhd{A_0}{\Gamma\cdot R}$.
\end{lemma}

\begin{proof}
Any path in $G_t$ with at least $2r-2$ illegal turns contains a loop in $G_t$ which has length at least $\epsilon_0$. The lemma now easily follows.
\end{proof}

We find the following terminology helpful. Suppose that $G_t$, $t\in [a,b]$, is a folding path and that $\alpha$ is a nontrivial conjugacy class in $\free$. As mentioned earlier, the immersed loop $\alpha\vert G_t \to G_t$ consists of legal segments separated by illegal turns. We let $\alpha_t^{\mathrm{leg}}$ denote the subset of $\alpha\vert G_t$ consisting of maximal legal segments of length at least $3$, and we write $\mathrm{leg}(\alpha\vert G_t)$ for the length of $\alpha_t^{\mathrm{leg}}$. This is the \define{legal length} of $\alpha\vert G_t$. The complement $\alpha\vert G_t - \alpha_t^{\mathrm{leg}}$ consists of finitely many disconnected segments, and we write $\mathrm{ilg}(\alpha\vert G_t)$ for the sum of the lengths of the components of $\alpha\vert G_t - \alpha_t^{\mathrm{leg}}$ that contain at least $\breve{m}+1$ illegal turns (counting the endpoints). This is the \define{illegal length} of $\alpha\vert G_t$. Finally we write $\mathrm{ntr}(\alpha\vert G_t)$ for the sum of the lengths of the remaining components of $\alpha\vert G_t - \alpha_t^{\mathrm{leg}}$, that is, those components with less than $\breve{m}+1$ illegal turns. This is the \define{neutral  length} of $\alpha\vert G_t$. By construction we thus have
\[\ell(\alpha\vert G_t) = \mathrm{leg}(\alpha\vert G_t) + \mathrm{ilg}(\alpha \vert G_t) + \mathrm{ntr}(\alpha \vert G_t).\]
Notice that, since every component of $\alpha_t^{\mathrm{leg}}$ has length at least $3$, there are at most $(\mathrm{leg}(\alpha\vert G_t)/3) +1$ components of $\alpha\vert G_t - \alpha_t^{\mathrm{leg}}$. On the other hand, each component contributing to $\mathrm{ntr}(\alpha\vert G_t)$ has length at most $3\breve{m}$ by definition, and so we find that
\[\mathrm{ntr}(\alpha\vert G_t) \le \breve{m}(\mathrm{leg}(\alpha\vert G_t) + 3).\]

The previous two lemmas allow us to show that the illegal length of $\alpha\vert G_t$ decreases exponentially fast along a folding path that remains close to the orbit of $\Gamma$.

\begin{lemma}[Illegal turn mortality rate]
\label{lem: unfolding_growth}
Suppose that $\Gamma$ is purely hyperbolic and that $\gamma\colon [a,b] \to \os$ is a folding path with $G_t = \gamma(t) \in  \nbhd{A_0}{\Gamma \cdot R}$ for all $t$. Then for every nontrivial conjugacy class $\alpha$ we have
\[\mathrm{ilg}(\alpha\vert\gamma(a)) \ge \frac{\epsilon_0 \breve{m}}{3(2r-2)(2\breve{m}+1)} \left(\frac{2\breve{m} + 1}{2\breve{m} }\right)^{\frac{(b-a)}{D_3}} \cdot \mathrm{ilg}(\alpha\vert \gamma(b)),\]
where $\epsilon_0$ is the minimal injectivity radius of any point in $\nbhd{A_0}{\Gamma\cdot R}$, $r = \rank(\free)$,  and $D_3$ is the constant from \Cref{lem: dont_survive}.
\end{lemma}

\begin{proof}
Let $p_b$ be a component of $\alpha\vert G_b - \alpha_b^{\mathrm{leg}}$ contributing to $\mathrm{ilg}(\alpha\vert G_b)$, and write $p_t$ for the corresponding path in $\alpha\vert G_t$ (i.e., $p_{t'}$ unfolds to $p_t$ for $t \le t'$). First note that for $t\in [a,b]$, the hypotheses on $p_b$ imply that every legal subsegment of $p_t$ has length less than $3$ (since legal segments of length at least $3$ grow under folding) and the number of illegal turns in $p_t$ is at least $\breve{m} +1$ (since $k(p_t)$ is nonincreasing in $t$). 

Suppose that $t\in [a,b]$ is such that $t-D_3\in [a,b]$. Partition $p_t$ into $s+1$ subpaths
\[p_t = p_t^1 \cdot \dotsb \cdot p_t^s \cdot q_t,\]
where each $p_t^i$ has $\breve{m}+1$ illegal turns and $q_t$ has less than $\breve{m}+1$ illegal turns (counting endpoints). Thus the number of illegal turns in $p_t$ is $k(p_t) = s\breve{m} + k(q_t)$, where in the case that $q_t$ is degenerate we view it as a segment with $1$ illegal turn so that $k(q_t) = 1$. By our the condition on $q_t$ and the assumption that $k(p_t) \ge (\breve{m}+1)$, it follows that
\[\frac{k(p_t)}{\breve{m}} \le s + 1\qquad\mathrm{and}\qquad 1 \le \frac{k(p_t)-1}{\breve{m}} \le \frac{k(p_t)}{2\breve{m}}.\]
Unfolding these $p_t^i$ to subsegments of $p_{t-D_3}$ and applying \Cref{lem: dont_survive}, we conclude that the number of illegal turns in each subsegment increases by at least $1$. Thus
\begin{eqnarray*}
k(p_{t - D_3}) &\ge& s(\breve{m}+1) + k(q_t) = k(p_t) + s
\ge k(p_t) + \frac{k(p_t)}{\breve{m}} - 1 \\
&\ge& k(p_t) + \frac{k(p_t)}{2\breve{m}} = \left( \frac{2\breve{m} + 1}{2\breve{m}} \right)k(p_t).
\end{eqnarray*}
So long as $a \le t-nD_3 \le b$, we may inductively apply this argument to conclude that 
\[k(p_a) \ge k(p_{t - n D_3}) \ge \left( \frac{2\breve{m}+1}{2\breve{m}} \right)^n k(p_t).\]
Using \Cref{lem: turns_to_length} to compare lengths with number of illegal turns, we conclude that
\begin{eqnarray*}
\ell(p_a) &\ge& \left(\frac{\epsilon_0}{4r-4} \right) k(p_{a}) \\
&\ge& \left(\frac{\epsilon_0}{4r-4} \right) \left( \frac{2\breve{m}+1}{2\breve{m}} \right)^{\floor{(b-a)/D_3}} k(p_t) \\
&\ge&  \left(\frac{\epsilon_0 \breve{m}}{3(2r-2)(2\breve{m}+1)} \right) \left( \frac{2\breve{m}+1}{2\breve{m}} \right)^{\frac{b-a}{D_3}} \ell(p_b).
\end{eqnarray*}
Summing these estimates over each component of $\alpha\vert G_b$ contributing to $\mathrm{ilg}(\alpha\vert G_b)$ gives the desired result.
\end{proof}

There is a similar estimate for the growth of legal length in the forward direction.

\begin{lemma}\label{lem:folding growth}
For any folding path $G_t$, $t\in [a,b]$, every nontrivial conjugacy class $\alpha\in \free$ satisfies
\[\mathrm{leg}(\alpha\vert G_b) \ge \mathrm{leg}(\alpha\vert G_a)\left(\frac{1}{3}\right)e^{b-a}.\]
\end{lemma} 
\begin{proof}
Let $p_a$ be a component of $\alpha_a^{\mathrm{leg}}$ and let $p_b$ be the corresponding segment in $\alpha\vert G_b$ (so that $p_b$ unfolds to $p_a$). Then $\ell(p_a) \ge 3$ by assumption, so Corollary 4.8 of \cite{BFhyp} gives
\[\ell(p_b) \ge 2 + (\ell(p_a) - 2) e^{b-a} \ge \ell(p_a) \left(1 - \frac{2}{\ell(p_a)}\right) e^{b-a} \ge \frac{\ell(p_a)}{3} e^{b-a} \]
Summing over the segments contributing to $\mathrm{leg}(\alpha\vert G_a)$ now proves the claim.
\end{proof}

Combining these estimates easily leads to uniform flaring along folding paths that stay close to the orbit $\Gamma\cdot R$:

\begin{proposition}[Flaring in folding paths] \label{flaring_along_folding}
Suppose $\Gamma\le \Out(\free)$ is purely hyperbolic and that $R\in \os$ is such that $\Gamma\cdot R$ is $A_0$--QCX. Then for all $\lambda_1 \ge 1$ there exits $D_1 \ge 1$ such that the following holds: For any nontrivial conjugacy class $\alpha$, any folding path $\gamma\colon \I\to \os$ with $G_s = \gamma(s)\in \nbhd{A_0}{\Gamma\cdot R}$, and any parameters $t\in \R$ and $d\ge D_1$ satisfying $[t -d,t+d]  \subset \I$  we have 
\[\lambda_1 \cdot \ell(\alpha\vert G_t) \le \max\left\{\ell(\alpha\vert G_{t-d}) , \ell(\alpha\vert G_{t+d})\right\}. \]
\end{proposition}
\begin{proof} Fix $t\in \I$.
\begin{case}[1] Suppose $\mathrm{ilg}(\alpha\vert G_t) \ge \ell(\alpha\vert G_t) / 2$. Then \Cref{lem: unfolding_growth} provides a constant $D'$ such that for all and $d \ge D'$ with $[t-d,t]\subset \I$ we have
\[\ell(\alpha\vert G_{t-d}) \ge \mathrm{ilg}(\alpha\vert G_{t-d}) \ge 2\lambda_1\cdot \mathrm{ilg}\ell(\alpha\vert G_t) \ge \lambda_1 \cdot\ell(\alpha\vert G_t).\]
\end{case}
\begin{case}[2] Suppose $\mathrm{ilg}(\alpha\vert G_t) < \ell(\alpha\vert G_t) / 2$ and $\mathrm{leg}(\alpha\vert G_t) \neq 0$. In this case we have
\begin{eqnarray*}
\ell(\alpha \vert G_t) &=& \mathrm{ilg}(\alpha\vert G_t) + \mathrm{leg}(\alpha\vert G_t) + \mathrm{ntr}(\alpha\vert G_t)\\
&\le& \tfrac{1}{2}\ell(\alpha\vert G_t) + \mathrm{leg}(\alpha\vert G_t) + \breve{m}(\mathrm{leg}(\alpha\vert G_t) + 3),
\end{eqnarray*}
which gives $\ell(\alpha\vert G_t) < 2(1+\breve{m})\mathrm{leg}(\alpha \vert G_t) + 6$. Note that $3 \le \mathrm{leg}(\alpha \vert G_t)$ by definition of (nonzero) legal length. \Cref{lem:folding growth} now provides a constant $D''$ such that for all $d\ge D''$ with $[t,t+d]\subset \I$ we similarly have
\[\ell(\alpha\vert G_{t+d}) \ge \mathrm{leg}(\alpha\vert G_{t+d}) \ge \lambda_1 4(1+ \breve{m}) \mathrm{leg}(\alpha \vert G_t) \ge \lambda_1\cdot \ell(\alpha \vert G_t).\]
\end{case}
\begin{case}[3] Suppose $\mathrm{ilg}(\alpha\vert G_t) < \ell(\alpha\vert G_t) / 2$ and $\mathrm{leg}(\alpha\vert G_t) = 0$. Then the above shows $\ell(\alpha \vert G_t) \le 6$. Thus by \Cref{not_short_for_long}, applied with $L_0 = 6\lambda_1 $ shows that there exists a constant $D_0$ so that for all $d\ge D_0$ we have $\ell(\alpha\vert G_{t+d}) > L_0 \ge \lambda_1 \ell(\alpha\vert G_t)$.\qedhere
\end{case}
\end{proof}

We are now prepared to prove the main result of this section:
\begin{proof}[Proof of \Cref{qcx_implies_flaring}]
Fix a finite generating set $S\subset \Gamma$ and a free basis $X$ of $\free$. We must produce $M\in \N$ such that for every nontrivial $\alpha\in \free$ and all $g_1,g_2 \in \Gamma$ with $\abs{g_i}_S \ge M$ and $\abs{g_1g_2} = \abs{g_1}_S +\abs{g_2}_S$ we have
\[2 \norm{\alpha}_X \le \max \left\{\norm{g_1(\alpha)}_X, \norm{g_2\inv(\alpha)}_X\right\}.\]

Recall first that, by \Cref{pro: distance}, there exists a constant $K = K(X,R)$ such that $\tfrac{1}{K}\norm{\alpha}_X \le \ell(\alpha\vert R) \le K\norm{\alpha}_X$ for every conjugacy class $\alpha$ in $\free$. We apply \Cref{flaring_along_folding} with $\lambda_1 = 2K^2e^{2A}$ and obtain a corresponding constant $D_1$. By \Cref{lem:qcx-implies_qi_in_os}, we know that $g\mapsto g\cdot R$ defines a quasi-isometric embedding of $(\Gamma,d_\Gamma)$ into $(\os,d_\os)$. Thus we may choose $M\in \N$ sufficiently large so that every $g\in \Gamma$ with $\abs{g}_S \ge M$ satisfies $d_\os(R,g\cdot R) \ge D_1 + 2A$. We claim that $\Gamma$ has $(2,M)$--conjugacy flaring.

Let $g_1,g_2\in \Gamma$ be any elements with $\abs{g_i}_S\ge M$ and $\abs{g_1g_2}_S = \abs{g_1}_S+\abs{g_2}_S$. It follows that there exists a geodesic $(h_{-k},\dotsc,h_{-1},e,h_1,\dotsc,h_{j})$ in $\Gamma$ with $h_{-k} = g_1\inv$ and $h_j = g_2$. In particular, $k = \abs{g_1}_S$ and $j = \abs{g_2}_S$. Since $\Gamma\cdot R$ is $A$--QCX by hypothesis, there exists a folding path $\gamma\colon \I\to \os$ that has Hausdorff distance at most $A$ from the image of $(h_{-k},\dotsc,h_j)$. Writing $G_t = \gamma(t)$, we may thus choose times $a < s < b$ in $\I$ so that
\begin{align}\label{eqn:distance_estimate}
&\dsym(G_a, h_{-k}\cdot R) \le A,  \\
 &\dsym(G_s, R) \le A, \notag \\
 &\dsym(G_b, h_j\cdot R) \le A. \notag
\end{align}
Since $\abs{g_1\inv}_S,\abs{g_2}_S \ge M$, the above remarks imply that
\begin{align*}
&d_\os(h_{-k}\cdot R, R) = d_\os(R, g_1\cdot R) \quad\text{and} \\
&d_\os(R,h_j\cdot R) = d_\os(R,g_2\cdot R)
\end{align*}
are both bounded below by $D_1+2A$. Thus by the triangle inequality we have
\begin{align*}
&d_\os(G_a,G_s)\ge D_1\quad \text{and} \\
&d_\os(G_s, G_b)\ge D_1,
\end{align*}
which is equivalent to $s-a \ge D_1$ and $b-s \ge D_1$. Since the folding path $\gamma(\I)$ lies in $\nbhd{A}{\Gamma\cdot R}$ and the orbit $\Gamma\cdot R$ is $A$--QCX by assumption, \Cref{flaring_along_folding} now ensures that
\[ 2K^2 e^{2A}\cdot \ell(\alpha\vert G_s)  \le \max\left\{ \ell(\alpha\vert G_a), \ell(\alpha \vert G_b)\right\}\]
for every nontrivial $\alpha\in \free$. Finally, since $d_\os \le \dsym$, equation \eqref{eqn:distance_estimate} implies that
\begin{align*}
&\ell(\alpha\vert G_a) \le e^A \ell(\alpha\vert g_1\inv \cdot R),\\
&\ell(\alpha\vert G_b) \le e^{A} \ell(\alpha\vert g_2\cdot R), \quad \text{and}\\
&\ell(\alpha\vert R) \le e^A\ell(\alpha \vert G_s).
\end{align*}
Combining the above two estimates and using the rule $\ell(\alpha\vert g\inv\cdot R) = \ell(g(\alpha)\vert R)$ yields
\begin{align*}
2 \norm{\alpha}_X &\le 2K\ell(\alpha\vert R) \\
&\le \tfrac{1}{K}e^{-A} \max\left\{\ell(\alpha \vert G_a), \ell(\alpha \vert G_b) \right\}\\
& \le \tfrac{1}{K}\max\left\{\ell(g_1(\alpha)\vert R), \ell(g_2\inv(\alpha)\vert R)\right\}\\
&\le \max\left\{\norm{g_1(\alpha)}_X, \norm{g_2\inv(\alpha)}_X\right\}.
\end{align*}
Since this holds for every nontrivial $\alpha\in \free$, we have proved the claim.
\end{proof}

\section{The Cayley graph bundle of a free group extension}
\label{sec:metric_bundle}

Fix $\Gamma \le \Out(\free)$ with finite generating set $S = \{ s_1, \ldots, s_n \}$, and fix a free basis $X = \{x_1, \ldots, x_r\}$ for $\free$. Recalling that the extension $E_{\Gamma}$ is naturally a subgroup of $\Aut(\free)$, choose lifts $t_i \in \Aut(\free)$ of $s_i$ for each $1\le i \le n$ so that $E_{\Gamma}$ is generated as a subgroup of $\Aut(\free)$ by $W = \{i_{x_1}, \ldots i_{x_r},t_1, \ldots, t_n\}$. That is 
\[E_{\Gamma} = \langle i_{x_1}, \ldots i_{x_r},t_1, \ldots, t_n \rangle  \le \Aut(\free).\]
Here, $i_{x}$ is the inner automorphism given by conjugation by $x \in \free$, i.e., $i_x(\alpha) = x \alpha x^{-1}$ for $\alpha \in \free$. Note that by construction, 
\[t i_{x} t^{-1} = i_{t(x)} \in \Aut(\free)\]
for each $x \in \free$ and each $t \in \Aut(\free)$. For convenience, set $\hat{X} = \{ i_{x_1}, \ldots, i_{x_r} \}$ and $\hat\free = \langle \hat{X} \rangle$, so that $\hat \free$ is the image of $\free$ in $\Aut(\free)$. Note that $\hat\free$ is also the kernel of the homomorphism $E_{\Gamma} \to \Gamma$. In general, for $g \in \Gamma$ we denote a lift of $g$ to an automorphism in the extension $E_{\Gamma}$ by $\tilde{g}$.

Let $T = \cay{X}{\free}$,  $\bund = \cay{W}{E_{\Gamma}}$, and $\base = \cay{S}{\Gamma}$, where $\cay{\cdot}{\cdot}$ denotes the Cayley graph with the specified generating set equipped with the path metric in which each edge has length one. Set $\rose$ to be the standard rose on the generating set $X$ so that $\rose = T / \free$. There is an obvious equivariant simplicial map
\[p\colon \bund \to \base \]
obtained from the surjective homomorphism $E_{\Gamma} \to \Gamma$.  In details, $p\colon \bund \to \base$ is defined to be the homomorphism $E_{\Gamma} \to \Gamma$ on the vertices of $\bund$ and maps edges of $\bund$ to either vertices or edges of $\base$, depending on whether the edge corresponds to a generator in $X$ or $S$, respectively. Note that for each $b\in \Gamma$, the preimage $T_{b} = p^{-1}(b)$ is the simplicial tree (isomorphic to $T$) with vertices labeled by the coset $\tilde{b}\hat\free$ ($\tilde{b}$ any lift of $b$) and edges labeled by $\hat X$. We write $d_b$ for the induced path metric on the fiber $T_b$ over $b\in \Gamma$.

In Example $1.8$ of  \cite{MjSardar}, it is verified that $p\colon \bund \to \base$ is a metric graph bundle. We provide the details here for completeness. We first make the following observation.

\begin{lemma} \label{lem: distance_in_fiber}
Let $g_1,g_2$ be vertices of $p^{-1}(b) = T_b$. Then $g_1^{-1}g_2 \in \hat\free\cong \free$ and 
\[d_{b}(g_1,g_2) = \abs{g_1^{-1}g_2}_{\hat X}.\]
\end{lemma}

\begin{proof}
Since $T_b$ is a graph (it is a tree), $d_{b}(g_1,g_2)$ counts the minimal number of edges traversed by any path from $g_1$ to $g_2$ \emph{that remains in} $T_b$. Such a path consists of edges labeled by generators in $W$ coming from $\hat X$. As any such path represents $g_1^{-1}g_2$, we have $g_1\inv g_2\in \hat\free$ and $d_b(g_1,g_2)\ge\abs{g_1\inv g_2}_{\hat X}$. Conversely, writing $g_1^{-1}g_2$ in terms of $\{i_{x_1}^\pm,\dotsc, i_{x_r}^\pm\}$ produces a path in $T_b$ from $g_1$ to $g_2$. Thus $d_b(g_1,g_2)\le \abs{g_1\inv g_2}_{\hat X}$.
\end{proof}

\begin{lemma}\label{lem: its_a_bundle}
The equivariant map of Cayley graphs $p: \bund \to \base$ is a metric graph bundle.
\end{lemma}
\begin{proof}
For each $n\in \N$, the $n$--ball $\{g \in E_\Gamma : \abs{g}_W \le n\}$ is finite. We may therefore define the properness function $f\colon \N \to \N$ by setting $f(n) = \max \{\abs{i_\alpha}_{\hat X} : i_\alpha\in \hat\free \text{ and }\abs{i_\alpha}_W\le n\}$. Then for any $b\in \Gamma$ and any $g_1,g_2$ in $T_b = p^{-1}(b)$, \Cref{lem: distance_in_fiber} implies that
\[ d_{b}(g_1,g_2) = \abs{g_1\inv g_2}_{\hat X} \le f\left(\abs{g_1\inv g_2}_W\right)  = f\left( d_\bund(g_1,g_2)\right),\]
as required. Lastly, suppose $b_1,b_2\in \base$ are adjacent vertices and that $g_1\in T_{b_1}$ is any vertex over $b_1$. Then $b_2 = b_1s$ for some $s \in S$. If $t\in W$ is the chosen lift of $s$, then $g_1t$ is adjacent to $g_1$ in $\bund$ and satisfies $p(g_1 t) = b_1 s = b_2$, as desired. This completes the proof that $p\colon \bund \to \base$ is a metric graph bundle.
\end{proof}

Using our choice of generators in $W$, we may define canonical lifts of paths in $\base$ through any particular point in a fiber. For $N\in \N$, let $\gamma\colon[-N,N] \to \base$ be any edge path in $\base$ (by which we mean a path that maps each integer $j$ to a vertex and each intervening interval $[j,j+1]$ isometrically onto an edge) and let $\tilde{\gamma}(0)$ be any vertex in the fiber $T_{\gamma(0)}$. For each integer $-N\le j < N$, the product $s_j = \gamma(j)^{-1}\gamma(j+1)$ then lies in the generating set $S$, and we let $t_j$ be the chosen lift of $s_j$ to $W$. Thus for $j> 0$ we have $\gamma(j) = \gamma(0)s_0\dotsb s_{j-1}$ and $\gamma(-j) = \gamma(0)s_{-1}\inv\dotsb s_{-j}\inv$. Accordingly, the \define{canonical lift of $\gamma$ through $\tilde{\gamma}(0)\in T_{\gamma(0)}$} is defined to be the edge path $\tilde{\gamma}\colon [-N,N] \to \bund$ given by
\[\tilde{\gamma}(j) = \tilde{\gamma}(0)t_0\dotsb t_{j-1}\qquad\text{and}\qquad \tilde{\gamma}(-j) = \tilde{\gamma}(0)t_{-1}\inv\dotsb t_{-j}\inv.\]
for each integer $0 \le j \le N$. Observe that $p(\tilde{\gamma}(j)) = \gamma(j)$, so that $\tilde{\gamma}$ is in fact a lift of $\gamma$.  Moreover, since $p\colon \bund\to \base$ is $1$--Lipschitz, when the original path $\gamma\colon[-N,N] \to \base$ is a geodesic, so is the canonical lift of $\gamma$ through any point in $T_{\gamma(0)}$. These lifts will be instrumental in establishing the flaring property for the metric graph bundle $\bund\to \base$, which we do in \Cref{prop: bundle_flaring} below.

\section{Conjugacy flaring implies hyperbolicity of $E_{\Gamma}$} \label{sec: hyp_bundle}
In this section we complete the proof of our main theorem and show that the $\free$--extension group $E_\Gamma$ is hyperbolic when $\Gamma\le \Out(\free)$ is purely hyperbolic and qi-embeds into the factor complex $\fc$. We first show that conjugacy flaring for the group $\Gamma$ implies that the metric bundle $\bund\to \base$ defined in \Cref{sec:metric_bundle} has the flaring property. Combining with \Cref{thm:bundle_hyperbolicity}, this will show that $\bund$, and consequently $E_\Gamma$,  is hyperbolic.

\begin{proposition}[Conjugacy flaring implies the flaring property] \label{prop: bundle_flaring}
Suppose that a finitely generated subgroup $\Gamma\le \Out(\free)$ satisfies $(\lambda,N)$--conjugacy flaring for some $\lambda> 1$ and $N\in \N$. Then the corresponding metric graph bundle $p\colon \bund \to \base$ satisfies the flaring condition.
\end{proposition}
\begin{proof}
By hypothesis, there is a finite generating set $S = \{s_1,\dotsc,s_n\}$ of $\Gamma$ and a free basis $X = \{x_1,\dotsc,x_r\}$ of $\free$ with respect to which $\Gamma$ has $(\lambda,N)$--conjugacy flaring (see \Cref{sec: qcx_implies_flar}). As in \Cref{sec:metric_bundle} we then consider the generating set $W = \{i_{x_1},\dotsc,i_{x_r},t_1,\dotsc, t_n\}$ of $E_\Gamma$, where $t_i$ denotes a chosen lift of $s_i$, and the natural simplicial surjection $p\colon \bund\to \base$, where $\bund = \cay{W}{E_\Gamma}$ and $\base = \cay{S}{\Gamma}$. As before, set $\hat X$ equal to the subset of the generators of $W$ coming from $X$ and denote the isomorphic image of $\free$ in $E_{\Gamma}$ by $\hat \free = \langle \hat X \rangle$. 

To establish the flaring property, we must show that for every $k\ge 1$ there exists $\lambda_k > 1$ and $n_k,M_k\in \N$ such that for any geodesic $\gamma\colon[-n_k,n_k]\to \base$ and any two $k$--qi lifts $\tilde{\gamma}_1$ and $\tilde{\gamma}_2$ satisfying $d_{\gamma(0)}(\tilde{\gamma}_1(0),\tilde{\gamma}_2(0)) \ge M_k$ we have
\[\lambda_k \cdot d_{\gamma(0)}(\tilde{\gamma}_1(0),\tilde{\gamma}_2(0)) \le \max\left\{d_{\gamma(n_k)}(\tilde{\gamma}_1(n_k),\tilde{\gamma}_2(n_k)),\; d_{\gamma(-n_k)}(\tilde{\gamma}_1(-n_k),\tilde{\gamma}_2(-n_k))\right\}.\]
In fact, we show that in terms of the given conjugacy flaring constants $(\lambda,N)$ we may take $\lambda_k = \tfrac{\lambda+1}{2}$ and $n_k = N$ (each independent of $k$) so that given any $k\ge1$, if
\[M_k = 2(\lambda + 2e_k)/(\lambda-1)\]
 then the flaring condition holds with these constants. Here $e_k = f(N+1+kN+k)$, where $f(\cdot)$ is the properness function for the bundle $\bund \to \base$.

Let $\gamma\colon[-N,N] \to \base$ be a geodesic and set $b = \gamma(0)$. Suppose that two $k$--qi lifts $\tilde{\gamma}_1,\tilde{\gamma}_2\colon [-N,N] \to \bund$  are given (hence, $p(\tilde{\gamma}_i(j)) = \gamma(j)$ for $i =1,2$ and each integer $j$). Recall from \Cref{sec:metric_bundle} that $T_{\gamma(j)} = p^{-1}(\gamma(j))$ is a simplicial tree whose edges are labeled by the free basis $\hat X$ of $\hat \free$. With respect to this basis, the element $\tilde{\gamma}_1(0)^{-1}\tilde{\gamma}_2(0)\in \hat \free$ may not by cyclically reduced. However, there is some $x\in \hat X$ so that $i_\alpha = \tilde{\gamma}_1(0)^{-1}\tilde{\gamma}_2(0)x \in \hat \free$ is cyclically reduced. Then $i_\alpha$ has the property that $\norm{i_\alpha}_{\hat X} = \abs{i_\alpha}_{\hat X}$ and that $\abs{i_\alpha}_{\hat X}$ differs from $d_{b}(\tilde{\gamma}_1(0),\tilde{\gamma}_2(0)) = \abs{\tilde{\gamma}_1(0)\inv\tilde{\gamma}_2(0)}_{\hat X}$ by at most $1$. Set $z_1 = \tilde{\gamma}_1(0)$ and $z_2 = \tilde{\gamma}_2(0)x \in T_{b}$ so that by construction,
\[ z_1 i_\alpha= z_2. \]

For each integer $-N\le j < N$, let us set $s_j = \gamma(j)\inv\gamma(j+1)\in S$. Since $\gamma$ is a geodesic, the products 
\begin{align*}
&g = s_{-N}\cdots s_{-1}\in\Gamma \quad \text{and}\\
&h = s_0\dotsb s_{N-1}\in \Gamma
\end{align*}
satisfy $\abs{g}_S = \abs{h}_S = N$ and $\abs{gh}_S = \abs{g}_S + \abs{h}_S$. Therefore $(\lambda,N)$--conjugacy flaring implies that
\begin{align*}
\max\{ \norm{g(\alpha)}_X, \norm{h\inv(\alpha)}_X \} &\ge \lambda \cdot \norm{\alpha}_X\\
&= \lambda\cdot \abs{\alpha}_X  \\
&= \lambda \cdot \abs{i_\alpha}_{\hat X}\\
&\ge \lambda\cdot (d_{b}(\tilde{\gamma}_1(0),\tilde{\gamma}_2(0)) - 1).
\end{align*}

Let $\tilde{\gamma}_{z_1}, \tilde{\gamma}_{z_2}\colon [-N,N] \to \bund$ be the canonical (geodesic) lifts of $\gamma\colon[-N,N] \to \base$ through the points $z_1$ and $z_2$, respectively. Let us also write $\tilde{g} = t_{-N}\dotsb t_{-1}$ and $\tilde{h} = t_0\dotsb t_{N-1}$, where $t_i$ is the chosen lift of $s_i\in S$ in the generating set $W$ of $E_\Gamma$. By construction, $\tilde{g}$ and $\tilde{h}$ are also lifts of $g,h\in \Gamma\le\Out(\free)$ to $E_\Gamma\le \Aut(\free)$. Recall that the canonical lifts $\tilde{\gamma}_{z_j}$ are defined so that
\begin{align*}
&\tilde{\gamma}_{z_j}(-N) = z_j t\inv_{-1}\dotsb t\inv_{-N} = z_j \tilde{g}\inv \quad\text{and} \\
&\tilde{\gamma}_{z_j}(N) = z_j t_0\dotsb t_{N-1} = z_j \tilde{h}
\end{align*}
for $j=1,2$. Therefore
\begin{align*}
&\tilde{\gamma}_{z_1}(-N)^{-1} \tilde{\gamma}_{z_2}(-N) = (\tilde{g} z_1\inv) (z_2 \tilde{g}\inv) = \tilde{g}i_\alpha \tilde{g}\inv = i_{\tilde{g}(\alpha)} \quad \text{and}\\
&\tilde{\gamma}_{z_1}(N)^{-1} \tilde{\gamma}_{z_2}(N) = (\tilde{h}\inv z_1\inv) (z_2 \tilde{h}) = \tilde{h}\inv i_\alpha \tilde{h} = i_{\tilde{h}^{-1}(\alpha)}.
\end{align*}
Hence, the endpoints of our canonical lifts of $\gamma$ satisfy
\[d_{\gamma(-N)}(\tilde{\gamma}_{z_1}(-N), \tilde{\gamma}_{z_2}(-N)) = \abs{\tilde{\gamma}_{z_1}(-N)^{-1} \tilde{\gamma}_{z_2}(-N)}_{\hat X} = \abs{ i_{\tilde{g}(\alpha)}}_{\hat X} 
= \abs{\tilde{g}(\alpha)}_X \ge \norm{g(\alpha)}_X\]
and
\[d_{\gamma(N)}(\tilde{\gamma}_{z_1}(N), \tilde{\gamma}_{z_2}(N)) = \abs{\tilde{\gamma}_{z_1}(N)^{-1} \tilde{\gamma}_{z_2}(N)}_{\hat X} =  \abs{i_{\tilde{h}^{-1}(\alpha)}}_{\hat X}
= \abs{\tilde{h}\inv(\alpha)}_X \ge \norm{h\inv(\alpha)}_X.\]
In light of conjugacy flaring, it follows that we have
\[\max\left\{d_{\gamma(-N)}\big(\tilde{\gamma}_{z_1}(-N), \tilde{\gamma}_{z_2}(-N)\big),\; d_{\gamma(N)}\big(\tilde{\gamma}_{z_1}(N), \tilde{\gamma}_{z_2}(N)\big)\right\} \ge \lambda \cdot \big(d_b(\tilde{\gamma}_1(0),\tilde{\gamma}_2(0)) - 1\big).\]

Let us now estimate the distances between our canonical lifts $\tilde{\gamma}_{z_j}$ and the given lifts $\tilde{\gamma}_j$ of $\gamma$. By metric properness, for $j=1,2$ we have
\begin{eqnarray*}
d_{\gamma( N)}(\tilde{\gamma}_{z_j}(N), \tilde{\gamma}_{j} ( N))) &\le& f\Big(d_{\bund}(\tilde{\gamma}_{z_j}( N), \tilde{\gamma}_j ( N))\Big)\\
&\le& f\Big(d_{\bund}(\tilde{\gamma}_{z_j}(N), \tilde{\gamma}_{z_j}(0))+d_{\bund}(\tilde{\gamma}_{z_j}(0),\tilde{\gamma}_{j}(0) ) + d_{\bund}(\tilde{\gamma}_{j}(0),\tilde{\gamma}_{j}(N))\Big) \\
&\le& f(N+1 +kN+k) = e_k.
\end{eqnarray*}
We similarly have $d_{\gamma(-N)}(\tilde{\gamma}_{z_j}(-N), \tilde{\gamma}_j(-N)) \le e_k$ for $j=1,2$. The triangle inequality thus gives
\[ d_{\gamma(\ast)}(\tilde{\gamma}_1(\ast), \tilde{\gamma}_2(\ast)) \ge d_{\ast}(\tilde{\gamma}_{z_1}(\ast), \tilde{\gamma}_{z_2}(\ast)) - 2e_k\]
for $\ast = \pm N$. Combining with our above estimate, it follows that the given lifts $\tilde{\gamma}_1$ and $\tilde{\gamma}_2$  satisfy
\[\max\left\{d_{\gamma(-N)}\big(\tilde{\gamma}_1(-N), \tilde{\gamma}_2(-N)\big),\; d_{\gamma(N)}\big(\tilde{\gamma}_1(N), \tilde{\gamma}_2(N)\big)\right\} \ge \lambda \cdot d_b(\tilde{\gamma}_1(0),\tilde{\gamma}_2(0)) - \lambda - 2e_k.\]
Therefore whenever $d_b(\tilde{\gamma}_1(0),\tilde{\gamma}_2(0)) \ge M_k = 2(\lambda + 2e_k)/(\lambda-1) $, so that
\begin{eqnarray*}
\lambda \cdot d_b(\tilde{\gamma}_1(0),\tilde{\gamma}_2(0)) - \lambda - 2e_k &\ge& \lambda \cdot d_b(\tilde{\gamma}_1(0),\tilde{\gamma}_2(0)) - \tfrac{\lambda-1}{2} d_b(\tilde{\gamma}_1(0),\tilde{\gamma}_2(0))\\
&=& \tfrac{\lambda+1}{2} d_b(\tilde{\gamma}_1(0),\tilde{\gamma}_2(0)).
\end{eqnarray*}
we obtain the inequality required by the flaring property. This completes the proof.
\end{proof}

\begin{theorem}[Hyperbolic extensions]\label{thm:qcx_implies_hyperbolic}
Suppose that $\Gamma \le \Out(\free)$ is purely hyperbolic and that there exists $R\in \os$ so that $\Gamma\cdot R$ is $A$--QCX. Then the corresponding  extension group $E_{\Gamma}$ is hyperbolic.
\end{theorem}
\begin{proof}
Since $\bund$ is the Cayley graph of $E_{\Gamma}$, it suffice to show that $\bund$ is hyperbolic. We show that the metric graph bundle $\bund \to \base$ satisfies the three conditions for hyperbolicity appearing in \Cref{thm:bundle_hyperbolicity} (the Mj--Sardar Theorem). Conditions $(1)$ and $(2)$ are obvious since each fiber is isomorphic to the universal cover of an $\rank(\free)$--petal rose. Since the hypotheses imply that $\Gamma$ has conjugacy flaring (\Cref{qcx_implies_flaring}), condition $(3)$ follows from \Cref{prop: bundle_flaring}. Hence, $E_{\Gamma}$ is hyperbolic.
\end{proof}

\begin{corollary} \label{for: main_result}
Suppose $\Gamma\le \Out(\free)$ is purely hyperbolic and qi-embeds into $\fc$. Then $E_\Gamma$ is hyperbolic.
\end{corollary}
\begin{proof}
This follows immediately from \Cref{cor:fc-qi_implies_qcx} and \Cref{thm:qcx_implies_hyperbolic}.
\end{proof}


\section{Applications} \label{sec: apps}

In this section, we produce examples of hyperbolic extensions of the free group $\free$ using the main result of this paper. We begin by defining a version of the intersection graph $\int$ for $\free$, which is an $\Out(\free)$--graph introduced by Kapovich and Lustig in \cite{kapovich2007geometric}. First, let $\int'$ be the graph whose vertices are conjugacy class of $\free$ and two vertices are joined by an edge if there is a very small simplicial tree $\free \curvearrowright T$ in which each conjugacy class fixes a point. (Recall that a simplicial tree is very small if edge stabilizers are maximal cyclic and tripod stabilizers are trivial.) Define $\int$ to be the connected component of $\int'$ that contains the primitive conjugacy classes. We note that there is a coarsely Lipschitz surjective map $\Theta \colon \F \to \int$ given by mapping the free factor $A$ to the set of primitive conjugacy classes that are contained in $A$. Note that $\Theta: \F \to \int$ is $\Out(\free)$--equivariant. 

Second, recall that the action of a non-virtually cyclic group $G$ on a hyperbolic metric space $X$ is \define{WPD} if for every $g \in G$ with positive translation length on $X$, the following property holds: for every $R\ge 0$ and every $x \in X$ there is an $N\ge 1$ so that the set
\[ \left\{ \phi \in G : d_{X}(x, \phi(x)))\le R \quad \text{and} \quad d_{X}(g^N(x), \phi(g^N(x)))\le R \right\} \]
is finite. It is further required that the group $G$ contains an element that acts with positive translation length on $X$. This property was first defined by Bestvina--Fujiwara in \cite{bestvina2002bounded}, where it was shown that the action of the mapping class group on the curve complex is WPD. The following theorem was communicated to us by Patrick Reynolds. For complete proofs see Mann \cite{Mann-thesis} and \cite[Theorem 4.2, Proposition 4.4]{dowdall2016co}.

\begin{theorem}[Mann--Reynolds \cite{MR2}]\label{th: M-R}
The graph $\int$ is hyperbolic and $f \in \Out(\free)$ acts with positive translation length on $\int$ if and only if $f$ is atoroidal and fully irreducible. Moreover the action $\Out(\free) \curvearrowright \int$ is WPD.
\end{theorem}

Following Bestvina--Fujiwara, we say that loxodromic elements $f_1 ,f_2 \in G$ are \define{independent} if their quasigeodesic axes in $X$ do not contain rays that have finite Hausdorff distance from one another. Said differently, $f_1$ and $f_2$ are independent if they determine $4$ distinct points on the Gromov boundary of $X$. The WPD condition can be used to understand how distinct loxodromic elements can fail to be independent. In particular, Proposition $6$ of \cite{bestvina2002bounded}, implies that $f_1$ and $f_2$ are independent if and only if they do not have a common power. Since \Cref{th: M-R} states that the action $\Out(\free) \curvearrowright \int$ is WPD, two hyperbolic, fully irreducible automorphisms $f_1,f_2 \in \Out(\free)$ are independent if and only if they have no common power. Thus the notion of independence of two fully irreducibles (with respect to the action $\Out(\free) \curvearrowright \int$) is intrinsic to the algebra of $\Out(\free)$.

Using \Cref{th: M-R}, we have (a priori weaker) version of our main theorem:

\begin{theorem}\label{th: intersection_main}
Let $\Gamma \le \Out(\free)$ be a finitely generated subgroup such that some (any) orbit map into $\int$ is a quasi-isometric embedding. Then the corresponding extension $E_{\Gamma}$ is hyperbolic.
\end{theorem}

\begin{proof}
Fix $A \in \F$ and let $\mathcal{O}\colon \Gamma \to \F$ be the corresponding orbit map into the free factor complex. By assumption $\Theta \circ \mathcal{O}\colon \Gamma \to \int$ is a quasi-isometric embedding. Since $\Theta$ is coarsely Lipschitz, $\mathcal{O}$ must also be a quasi-isometric embedding. Moreover, since all outer automorphisms with positive translation length of $\int$ are hyperbolic, $\Gamma$ must be purely hyperbolic, i.e. each infinite order element is atoroidal. Now apply \Cref{for: main_result} to conclude that $E_{\Gamma}$ is hyperbolic.
\end{proof}

We remark that our subsequence work \cite{dowdall2016co} implies that \Cref{th: intersection_main} is equivalent to our main theorem \Cref{th: intro_main_1}.

Our first application is a new proof of the following theorem of Bestvina--Feighn--Handel \cite{BFHlam}, where we allow for any number of hyperbolic, fully irreducible automorphisms. 

\begin{theorem}
\label{thm:free_construction}
Let $f_1,\ldots, f_k \in \Out(\free)$ be a collection of pairwise independent, hyperbolic, fully irreducible outer automorphisms. Then for sufficiently large $N\ge 1$, every nonidentity element of 
\[\Gamma =  \langle f_1^N,\ldots ,f_k^N \rangle \]
is hyperbolic and fully irreducible. Moreover, $\Gamma$ is isomorphic to the free group of rank $k$ and the extension $E_{\Gamma}$ is hyperbolic.
\end{theorem}

\begin{proof}
The proof that the subgroup quasi-isometrically embeds into $\int$ follows from a standard geometric ping-pong argument for groups acting on hyperbolic spaces, exactly as in the proof of
Theorem $1.4$ (Abundance of Schottky groups) in Kent--Leininger \cite{KentLein}. 
One can also deduce the result from \cite[Lemma 3.2]{TT}.
The point is that we are dealing with a collection of independent loxodromic automorphisms of a hyperbolic graph. 
To conclude that $E_{\Gamma}$ is hyperbolic, apply \Cref{th: intersection_main}.
\end{proof}

Our next application, to the authors' knowledge, produces the first examples of hyperbolic $\free$--extensions $E_{\Gamma}$ where $\Gamma$ is has torsion and is not virtually cyclic. First, for a finite group $H \le \Out(\free)$ say that a hyperbolic, fully irreducible $f \in \Out(\free)$ is \define{independent for $H$} if  $f$ and $hfh^{-1}$ are independent for each $h \in H$. Hence, $f$ is independent for $H$ if and only if $H \cap \mathrm{comm}(f) = \emptyset$, where $\mathrm{comm}(f)$ is the commensurator of $f$ in $\Out(\free)$.

\begin{theorem}\label{th: v_free_examples}
Let $H$ be a finite subgroup of $\Out(\free)$ and let $f \in \Out(\free)$ be a hyperbolic, fully irreducible outer automorphism that is independent for $H$. Then for all sufficiently large $N \ge 1$,  the subgroup
 \[ \Gamma = \langle H, f^N \rangle \]
 is isomorphic to $H* \Z$ and  quasi-isometrically embeds into $\int$. Hence, the $\free$-by-$(H*\Z)$ extension $E_{\Gamma}$ is hyperbolic.
\end{theorem}

\begin{proof}
Fix $x \in \int$ and for each $h\in H$ set $f_h = hfh^{-1}$. Let $D = \max_{h \in H}d(x,hx)$. Consider the Cayley graph $\C_h$ of $\langle f_h \rangle$ and the equivariant orbit map $\C_h \to \int$ obtained by mapping $f_h^i$ to $f_h^i (hx)$ and edges to geodesic segments. Since $f$ has positive translation length on $\int$ by \Cref{th: M-R}, the maps $\C_h \to \int$ are all $K_0$--quasi-isometric embeddings (for some $K_0 \ge1$). Let us write $\rho_h^\pm\colon [0,\infty)\to \int$ for the positive and negative subrays of $\C_h\to \int$ based at $hx$. Since the $f_h$ for $h\in H$ are all pairwise independent, no distinct pair or rays in the set $\{\rho_h^+,\rho_h^-\}_{h\in H}$ have finite Hausdorff distance.

Similar to \cite{KentLein}, we now consider the following set of paths in $\int$. For any $h_1,h_2\in H$ and $\epsilon_1,\epsilon_2\in\{+,-\}$ with $\rho_{h_1}^{\epsilon_1}\ne \rho_{h_2}^{\epsilon_2}$, we may build a biinfinite path in $\int$ by traversing $\rho_{h_1}^{\epsilon_1}$ with the reverse parameterization, then following a geodesic from $h_1 x$ to $h_2 x$ (which has length at most $D$), and lastly traversing the ray  $\rho_{h_2}^{\epsilon_2}$ with the usual parameterization. 
As there are finitely many such paths and the chosen rays $\rho_{h_1}^{\epsilon_1}$ and $\rho_{h_2}^{\epsilon_2}$ have infinite Hausdorff distance, there exists a uniform constant $K_1\ge 1$ so that each of these paths is a $K_1$--quasigeodesic in $\int$. We call subpaths of these $K_1$--quasigeodesics, as well as their images under the isometric action of $\Out(\free)$ on $\int$, \emph{$f$-pieces}.

Since $\int$ is hyperbolic, there exist $L, K_2 \ge 1$ so that any $L$--local, $K_1$--quasigeodesic is a $K_2$--quasigeodesic \cite{BH}. In particular, if $\gamma\colon \I \to \int$ is any path that agrees with some $f$-piece on every length $L$ subinterval of $\I$, then $\gamma$ is a $K_2$--quasigeodesic.

Now take $N$ to be an integer larger than $L$ and $D$, and let $\theta \colon H * \Z \to \Out(\free)$ be the homomorphism that restricts to the identity on $H$ and maps the generator $t$ of $\Z$ to $f^N$. Let $\Gamma = \langle H,f^N \rangle$ be the image of this homomorphism and let $\C$ be the Cayley graph of $H * \Z$ for the generating set $\{t,h:h\in H\}$, metrized so that each edge labeled $h \in H$ has length $D$ and each edge labeled $t$ has length $N$ . We define a $\theta$--equivariant map $\mathcal{O} \colon \C \to \int$ as follows: For each vertex $w \in H *\Z$ of $\C$, we set $\mathcal{O}(w) = \theta(w)x$.  For $h \in H$, the edge in $\C$ from $1$ to $h$ is mapped by $\mathcal{O}$ to any geodesic from $x$ to $hx$, and the edge in $\C$ from $1$ to $t$ is mapped by $\mathcal{O}$ to the $f$-piece from $x$ to $f^Nx$ using the parameterization coming from the quasigeodesic $\C_1\to \int$. Now extend $\mathcal{O}$ by equivariance. Observe that for all $a,b\in \Z$ and $h\in H$, $\mathcal{O}$ maps the length $(a+b)N+D$ path in $\C$ from $1$ to $t^{a}ht^{b}$ to an $f$-piece in $\int$. Thus by construction, $\mathcal{O}$ maps any geodesic in $\C$ to a path that agrees with $f$-pieces on all subintervals of length at most $L$. Using the constant $K_2$ obtained above, it follows that $\mathcal{O}$ sends every geodesic path in $\C$ to a $K_2$--quasigeodesic in $\int$ and thus that $\mathcal{O}\colon \C\to \int$ is a $K_2$--quasi-isometric embedding. Since the metric on $\C$ differs from the word metric on $H * \Z$  (with our chosen generators) by a multiplicative factor of no more than $N$, we conclude that $\theta \colon H*\Z \to \int$ is an $N K_2$--quasi-isometric embedding. 

Finally, to see that $\theta$ is an isomorphism, note that $\theta$ itself is a quasi-isometric embedding into $\Out(\free)$. This is a simple consequence of the fact that any orbit map from $\Out(\free)$ to $\int$ is coarsely Lipschitz. Hence, $\theta$ must have finite kernel. Since each finite order $g\in H*\Z$ is conjugate into $H$, and $H$ injects into $\Gamma\le\Out(\free)$, we must have that $\theta \colon H*\Z \to \Gamma$ is an isomorphism.  Since \Cref{th: intersection_main} implies that $E_{\Gamma}$ is hyperbolic, this completes the proof.
\end{proof}

\begin{remark}
Note that for $\Gamma = \langle H , f^N \rangle \cong H *\Z$ as in \Cref{th: v_free_examples}, the subgroup
\[\Gamma_0 = \langle H, f^NHf^{-N} \rangle\]
is undistorted and isomorphic to $H * H$. Hence, the $\free$-by-$(H\ast H)$ extension $E_{\Gamma_0}$ is also hyperbolic. In the situation of surface group extensions, Honglin Min has constructed convex cocompact subgroups of the mapping class group that are isomorphic to the free product of two finite groups \cite{min2008hyperbolic}.
\end{remark}

Finally, we show how to construct examples of hyperbolic, fully irreducible $f \in \Out(\free)$ that are independent for a given finite group $H \le \Out(\free)$. First, say that the finite group $H \le \Out(\free)$ is \define{projectively good} if its image under the surjective homomorphism $\Out(\free) \to \mathrm{GL}_r(\Z)$ does not contain $-I$ (where $r = \rank(\free)$). 
Note that any finite group $H$ embeds into the outer automorphism group $\Out(\free(H))$ with projectively good image, where $\free(H)$ is the free group on $H$. This may be achieved by using the left action of $H$ on itself to embed $H$ into $\Aut(\free(H))$ as permutation automorphisms whose images in $\mathrm{GL}_{\abs{H}}(\Z)$ are permutation matrices.

\begin{example} \label{ex: torsion}
Let $H$ be any projectively good, finite subgroup of $\Out(\free)$ with $\rank(\free) \ge 3$. We show that there is a hyperbolic, fully irreducible $f \in \Out(\free)$ that is independent for $H$. By \Cref{th: v_free_examples}, this shows that there is a hyperbolic group $G$ fitting into the exact sequence
\[ 1 \longrightarrow \free \longrightarrow G \longrightarrow H*\Z \longrightarrow 1.\]
As any finite group embeds into the outer automorphism group of some free group with projectively good image, this shows that there exists extensions of the above form for any finite group $H$.

Suppose that $H \le \Out(\free)$ is a finite, projectively good subgroup. Write $r = \rank(\free)$. As in \Cref{lem: phyp_implies_finiteprobs}, the restriction of the homomorphism $\Out(\free) \to \mathrm{GL}_r(\Z)$ to $H$ is injective and we identify $H$ with its image in $\mathrm{GL}_r(\Z)$. 
\begin{claim}
\label{claim: mostow}
There is a matrix $A \in \mathrm{GL}_r(\Z)$ such that for any $h \in H \setminus 1$, the matrices $hAh^{-1}$ and $A$ have no common power. 
\end{claim}
We complete the argument before proving the claim. Let $A$ be a matrix as in the claim. Now an application of the main result of Clay--Pettet \cite{clay2012current} implies that there is a hyperbolic, fully irreducible  outer automorphism $f$ whose image in $\mathrm{GL}_r(\Z)$ is $A$. We then have that $f$ is independent for the finite group $H$. Otherwise, there is an $h \in H \setminus 1$ and integers $r,s$ such that $hf^rh^{-1} = f^s$. Applying the homomorphism $\Out(\free) \to \mathrm{GL}_r(\Z)$ we see that this equation contradicts our choice of $A$. Hence, $f$ is independent for $H$. To complete the example, it now suffices to prove the claim.

\begin{proof}[Proof of Claim \Cref{claim: mostow}]
By assumption, the finite subgroup $H \le \mathrm{GL}_r(\Z)$ does not contain $-I$. Hence, the action $H \curvearrowright \mathbb{RP}^{r-1} $ is effective and if we denote the fixed subspace of $h\in H$ by $V_h$, we have that $V_H = \cup_{h \in H \setminus 1}V_h$ is a union of positive-codimension projective hyperplanes. Hence, $\mathbb{RP}^{r-1} \setminus V_H$ is open.

Now let $B \in \mathrm{GL}_r(\Z)$ be the block diagonal matrix consisting of  $\left(\begin{smallmatrix}2 & 1\\1&1\end{smallmatrix}\right)$ in the upper left $2\times 2$ corner and  the identity matrix in the lower right corner. The eigenvalues for $B$ are $\lambda$, $1$, and $\lambda\inv$, where $\lambda$ is the golden ratio. Moreover, the $\lambda$--eigenspace is one-dimensional and so defines a point $[v]\in \mathbb{RP}^{r-1}$. Since $\mathbb{RP}^{r-1} \setminus V_H$ is open and every orbit of $\mathrm{GL}_r(\Z) \curvearrowright \mathbb{RP}^{r-1}$ is dense \cite[Lemma 8.5]{mostow1973strong}, there is a $C \in \mathrm{GL}_r(\Z)$ so that $C[v] \notin V_H$. Setting $A = CBC^{-1}$, we see that the $\lambda$--eigenspace of $A$ is one-dimensional and is not projectively fixed by any $h \in H \setminus 1$. Hence, no power of $A$ can equal any power of $hAh^{-1}$ for $h\in H\setminus 1$. This completes the proof of the claim.
\end{proof}

\end{example}


\bibliographystyle{alphanum}
\bibliography{freeConvexCoCpct}

\bigskip

\noindent
\begin{minipage}{.6\linewidth}
Department of Mathematics\\
Vanderbilt University\\
1326 Stevenson Center\\
Nashville, TN 37240\\
E-mail: {\tt spencer.dowdall@vanderbilt.edu}
\end{minipage}
\hfill
\begin{minipage}{.4\linewidth}
Department of Mathematics\\ 
Yale University\\ 
10 Hillhouse Avenue\\
New Haven, CT 06511, U.S.A\\
E-mail: {\tt s.taylor@yale.edu}
\end{minipage}

\end{document}